\input ./style/arxiv-general.cfg
\documentclass[MSNbibl,number,citesort,seceqn,dvips]{arxbj}
\makeatletter
   \@ifpackageloaded{graphicx}{}{\usepackage{graphicx}}
\makeatother
\aid{0}
\volume{21}
\issue{3}
\pubyear{2015}
\firstpage{1884}
\lastpage{1910}
\doi{10.3150/14-BEJ630} % kopijuoti is 'New paper accepted'
\makeatletter
\newcommand{\eqref}[1]{(\ref{#1})}
\newcommand{\N}{\mathbb{N}}

\newproclaim{definition}{Definition}[section]
\newtheorem{theorem}[definition]{Theorem}
\newtheorem{lemma}[definition]{Lemma}
\newtheorem{lemmas}{Lemma}[section]
\newremark{rem}[definition]{Remark}
\newremark{rems}[definition]{Remarks}
\newremark{example}[definition]{Example}
\newtheorem{corollary}[definition]{Corollary}
\newproclaim{assumption}[definition]{Assumption}

\renewcommand{\P}{{\mathbb{P}}}
\renewcommand{\S}{\mathcal{S}}
\renewcommand{\epsilon}{\varepsilon}
\makeatother

\begin{document}
\begin{frontmatter}

\title{Capacity of an associative memory model on random graph architectures}
\runtitle{Associative memory model on random graphs}

\begin{aug}
%%%% inicialai - be tarpu
\author[a]{\inits{M.}\fnms{Matthias}~\snm{L\"owe}\thanksref{a}\ead[label=e1]{maloewe@math.uni-muenster.de}}
\and
\author[b]{\inits{F.}\fnms{Franck}~\snm{Vermet}\corref{}\thanksref{b}\ead[label=e2]{franck.vermet@univ-brest.fr}}
%\author{\inits{}\fnms{}~\snm{}\thanksref{1}\ead[label=e1]{}}% \and
%\author{\inits{}\fnms{}~\snm{}\thanksref{}\ead[label=e2]{}}
%\author{\inits{}\fnms{}~\snm{}}
%%\runauthor{} %% auto
%\dedicated{}
\address[a]{Fachbereich Mathematik und Informatik, Universit\"at M\"
unster,
Einsteinstra\ss e 62, 48149 M\"unster, Germany. \printead{e1}}
\address[b]{Laboratoire de Math\'ematiques, UMR CNRS 6205,
Universit\'e de Bretagne Occidentale, 6 avenue Victor Le Gorgeu, CS
93837,
F-29238 Brest Cedex 3, France. \printead{e2}}
%\address[1]{\printead{e1}}
%\address[]{}
\end{aug}

% HISTORY:
\received{\smonth{1} \syear{2014}}
%\revised{\smonth{} \syear{}}

% ABSTRACT
%
\begin{abstract}
We analyze the storage capacity of the Hopfield models on classes of
random graphs. While such a setup has been analyzed for the case that
the underlying random graph model is an Erd\"os--Renyi graph, other
architectures, including those investigated in the recent neuroscience
literature, have not been studied yet. We develop a notion of storage
capacity that highlights the influence of the graph topology and give
results on the storage capacity for not too irregular random graph
models. The class of models investigated includes the popular power law
graphs for some parameter values.
\end{abstract}

% KEYWORDS
% visi is mazosios raides ir pagal abecele
%
\end{frontmatter}

%s1 #&#
\section{Introduction}\label{sec1}
Thirty years ago, in 1982, Hopfield introduced a toy model for a brain
that renewed the interest in neural networks and has nowadays become
popular under the name Hopfield model \cite{Hopfield1982}. This model
in its easiest version assumes that the neurons are fully connected and
have Ising-type activities, that is, they take the values $+1$, if a
neuron is firing and $-1$, if it is not, and is based on the principles
of statistical mechanics. Since Hopfield's ground-breaking work, it has
stimulated a large number of researchers from the areas of computer
science, theoretical physics and mathematics.

In the latter field, the Hopfield model is particularly challenging,
since it also can be considered as a spin glass model and spin glasses
are notoriously difficult to study. A survey over the mathematical
results in this area can be found in either \cite{bovierbook} or \cite{talabook}. It is worth mentioning that even in the parameter region
where no spin glass phase is expected, the Hopfield model still has to
offer surprising phenomena such as in \cite{gentzloewe}.

When being considered as a neural network, one of the aspects that have
been discussed most intensively is its so-called storage capacity.
Here, one tries to store information, so-called patterns in the model,
and the question is, how many patterns can be successfully retrieved by
the network dynamics, that is, how much
information can be stored in a model of $N$ neurons. One of the early
mathematical results states that if the patterns are independent and
identically distributed (i.i.d. for short) and consist of i.i.d. spins
and if their number $M$ is bounded by $\frac{1}2 N/\log N$, the patterns
can be recalled (see \cite{MPRV}) with probability converging to one as
$N \to\infty$ and that the constant $\frac{1}2$ is optimal (see \cite{Bov97}). Similar results hold true, if one starts with a corrupted
input -- if more than fifty percent of the input spins are correct, one
still is able to restore the originally ``learned'' patterns. However,
if one also allows for small errors in the retrieval of the patterns
one obtains a storage capacity of $M=\alpha N$ for some value of
$\alpha
$ smaller than $0.14$ (see \cite{newman,loukianova,talagrand}). This latter result is in agreement with both, computer
simulations as well as the predictions of the non-rigorous replica
method from statistical physics (see \cite{AGS}).

The setup of the Hopfield model has been generalized in various
aspects, for example, the condition of the independence has been
relaxed (see \cite{L98,LV_05}), patterns with more than two
spins values have been considered (see \cite{Piccoetal,LV_BEG,LV_05}), and Hopfield models on Erd\"os--Renyi graphs
were studied \cite{BG92,BG93,talagrand,LV10}.
The present paper starts with the observation that even though being
more general than the complete graph, also Erd\"os--Renyi graphs do not
seem to be the favorite architectures for a brain for scientists
working in neurobiology. There, the standard paradigm currently is
rather to model the brain as a small world graph (see \cite{bassett,Rubinov}). We will focus on the question, how many patterns can
be stored in a Hopfield model on a random graph, if this graph is no
longer necessarily an Erd\"os--Renyi graph.
{The classical notion of storage capacity requires that the patterns
are fixed points of the retrieval dynamics, that is, local minima of
the energy landscape of the Hopfield model (or, in \cite{newman,loukianova,talagrand}, not too far apart from such minima). It
turns out that this notion is already sensitive to the architecture of
the network \cite{LV10}. So it is conceivable that
there is a major influence of the underlying graph structure on the
model's capability to retrieve corrupted information. Associativity of
a network can be described as the potential to repair corrupted
information. We will therefore work with a notion of storage capacity
that takes this ability into account.} Moreover,
the relationship between network connectivity and the performance of
associative memory models has already been investigated in computer
simulations (see, e.g., \cite{ChenAdams}). Therefore, the goal
of the present note is to establish rigorous bounds on the storage
capacity of the Hopfield model on a wide class of random graph models,
where we interpret ``storage'' as the ability to retrieve corrupted information.
Similar questions have been addressed for the complete graph by
Burshtein \cite{Burshtein}.

We organize the paper in the following way: Section~\ref{sec2} introduces the
basic model we will be working with in the present paper. It also
addresses the question, what exactly we mean when talking about the
storage of patterns. Section~\ref{sec3} contains the main result of this paper.
The number of patterns one is able to store in the sense, that
a number of errors that is proportional to N can
be repaired by $\mathcal{O}(\log N))$ steps of the retrieval
dynamics is of order $\operatorname{const.}(\lambda_1)^2/ (m\log N)$, where
$\lambda_1$ is the largest eigenvalue of the adjacency matrix of
the graph and $m$ its maximal degree. A main ingredient of the proof is
thus to analyze the spectrum of the adjacency matrix of the graph that
serves as a model of the network architecture. This analysis is
provided in Section~\ref{sec4}.
Eventually, Section~\ref{sec5} contains the proof of the main result. An
\hyperref[app]{Appendix} will contain estimates on the minimum and maximum degree of an
Erd\"os--R\'enyi graph. These are needed to apply our main result to
the setting of such random graphs and may also be of independent interest.

%s2 #&#
\section{The model}\label{sec2}
The Hopfield model is a spin model on $N\in\N$ spins. $\sigma\in
\Sigma_N:= \{-1,+1\}^N$ describes the neural activities of $N$ neurons.
The information to be stored in the model are patterns $\xi^1, \ldots,
\xi^M \in\{-1,+1\}^N$. As usual, we will assume that these patterns
are i.i.d. and consist of i.i.d. spins $(\xi_i^\mu)$ with
\[
\P\bigl(\xi_i^\mu= \pm1\bigr)= \tfrac12.
\]
Note that $M$ may and in the interesting cases will be a function of
$N$. The architecture of the Hopfield model is an undirected graph
$G=(V,E)$, where $V={1, \ldots, N}$. With the help of the~patterns and
the graph, one defines the sequential dynamics $S=T_N \circ
T_{N-1}\circ\cdots\circ T_1$ and
the parallel dynamics $T=(T_i)$ on $\Sigma_N$. By definition $T_i$ only
changes the $i$th coordinate of a configuration $\sigma$ and
\begin{eqnarray*}
&&S(\sigma)=T_N \circ T_{N-1}\circ\cdots\circ
T_1(\sigma), \qquad T(\sigma )= \bigl(T_1(\sigma),
\ldots, T_N(\sigma)\bigr),
\\
&&\quad\mbox{with } T_i(\sigma)= \mathop{\operatorname{sgn}}
\Biggl(\sum_{j=1}^N \sigma_j
a_{ij} \sum_{\mu=1}^M
\xi_i^\mu\xi_j^\mu\Biggr)
\end{eqnarray*}
(with the convention that $\mathop{\operatorname{sgn}}(0)=1$, e.g.).
Here, $a_{ij}=a_{ji}=1$ if the edge between $i$ and $j$ is in $E$ and
$a_{ij}=a_{ji}=0$ otherwise.
%Hence $T_i$ flips the $\sigma_i$ if and only if this lowers the energy.
The dynamics can be thought of as governing the evolution of the system
from an input toward the nearest learned pattern. $\xi^\mu$ being a
fixed point of $S$ (or~$T$) can thus be interpreted as recognizing a
learned pattern. However, this is not really what one would call an
associative memory. An important feature of the standard Hopfield model
(the one where $G=K_N$, the complete graph on $N$ vertices) is thus
also that under certain restrictions on~$M$ (and the number of
corrupted neurons), with high probability, a corrupted version of $\xi
^\mu$, say $\tilde\xi^\mu$ converges to $\xi^\mu$ when being evolved
under the dynamics. This observation is also crucial for the present paper.

We can associate Hamiltonians (or energy functions) to these dynamics by
\[
H_N^{{ S}}(\sigma)= - \operatorname{Const.}(N) \sum
_{i,j=1}^N\sigma_i \sigma _j
a_{ij} \sum_{\mu=1}^M
\xi_i^\mu\xi_j^\mu
\]
and
\[
H_N^{{ T}}(\sigma)= - \operatorname{Const.}(N) \sum
_{i=1}^N \Biggl| \sum_{j=1}^N
\sigma_j a_{ij} \sum_{\mu=1}^M
\xi_i^\mu\xi_j^\mu\Biggr|,
\]
such that the energy will decrease along each trajectory of the dynamics:
\[
H_N^{{ S}}\bigl(S(\sigma)\bigr) \leq H_N^{{ S}}
\bigl((\sigma)\bigr) \quad\mbox{and}\quad H_N^{{
T}}\bigl(T(
\sigma)\bigr) \leq H_N^{{ T}}\bigl((\sigma)\bigr).
\]
The constant is chosen in such a way that the mean free energy of the
model is finite and not constantly equal to zero.

One can easily prove that the sequential dynamics will converge to a
fixed point of $S$ and that every fixed point of $S$ is a local minimum
of $H_N^{{ S}}$. In the parallel case, the dynamics $T$ will converge
to a fixed point or a 2-cycle of $T$.

The idea of this setup is that the patterns (as well as their negatives
$-(\xi^\mu), \mu=1, \ldots, M$)
are hopefully possible limits of the dynamics. For instance, this is
easily checked, if $M\equiv1$ and $G$ is the complete graph, that $\xi
^1$ is a local minimum of $H_N^S$, since then
\[
H_N^S(\sigma)=- \operatorname{Const.}(N) \Biggl(\sum
_{i=1}^N \sigma_i \xi
_i^1\Biggr)^2 + \operatorname{Const.}_1(N)
\]
and hoped to be inherited by the more general model, as long as $M$ is
small enough.
Indeed, for $M=1$, the stored pattern $\xi^1$ is still a local minimum
of $H_N^S$, if $G$ is only connected. In this case, one obtains that
\[
H_N^S(\sigma)=- \operatorname{Const.}(N) \sum
_{i,j=1}^N\sigma_i \sigma_j
a_{ij} \xi_i^1 \xi_j^1=
- \operatorname{Const.}(N) X^t A X
\]
with $X=(\sigma_i \xi_i^1)$ and $A=(a_{ij})$. From here, the assertion
is immediate (we are grateful to an anonymous referee for this remark).

%If this is indeed the case, then not only the patterns are points in $
%\Sigma_N$ have a large weight under the (random) Gibbs measure $$\mu_N(
%\sigma):=\frac{\exp(-\beta H_N(\sigma))}{Z_N},  \beta>0 $$
%(where $Z_N$ is a normalizing constant), but also they are fixed
%points of the sequential gradient descent dynamics $T=(T_i)$ on $
%\Sigma_N$.

When considering the stability of a random pattern $\xi^\mu$ under $S$
or $T$ in the above setting, we need to check whether
$
T_i(\xi^\mu)= \xi_i^\mu
$
holds for any $i$. Now
\[
T_i\bigl(\xi^\mu\bigr)= \mathop{\operatorname{sgn}}
\Biggl(\sum_{j=1}^N a_{ij} \sum
_{\nu=1}^M \xi _i^\nu
\xi_j^\nu\xi_j^\mu\Biggr)=
\mathop{\operatorname{sgn}}\Biggl(\sum_{j=1}^N
a_{ij} \xi_i^\mu +\sum
_{j=1}^N a_{ij}\sum
_{\nu\neq\mu} \xi_i^\nu\xi_j^\nu
\xi_j^\mu\Biggr).
\]
That is, we have a signal term of strength $d(i)$, the degree of vertex
$i$ (given by the first summand on the right-hand side of the above
equation) and a random noise term. The first observation is that the
network topology enters via the degrees of the nodes. Indeed in such a
simple setup -- the stability of stored information -- the minimum
degree of the vertices is clearly decisive to compute the models's
storage capacity: in the case where a vertex $i$ has a small degree,
the noise term will exceed the signal term, except for a very small
number of stored patterns. However, it seems to be obvious that also
global aspects, for example, whether or not the graph is connected,
must play a role. This is confirmed if we are setting up a Hopfield
model on graph $G$ consisting of a complete graph $K_m$ (on the
vertices $1, \ldots, m$) and the graph
$K_{N-m}$ on the vertices $m+1, \ldots, N$ with $\log N \ll m \ll N$
and if we assume that these two subgraphs are disconnected or just
connected by one arc. Each of the vertices thus has at least degree $m$
and it can be computed along the lines of \cite{MPRV} or \cite{Petritis} that at least $\frac{m}{2 \log N}$ patterns can be stored as
fixed points of the dynamics. However, if we try to store one pattern,
for example, $\xi^1$ with $\xi_i^1=1$ for all $i=1, \ldots, N$, and
start with a corrupted input $\tilde\xi^1$ with
\[
\tilde\xi_i^1=\cases{ %
-1, &\quad $i \le m$,
\vspace*{2pt}
\cr
1, &\quad $m+1 \le i \le N$,}
\]
we see that
\[
T_i\bigl(\tilde\xi^1\bigr)=\tilde\xi_i^1.
\]
Hence, $\tilde\xi^1$ is a fixed point implying that the retrieval
dynamics is not able to correct $m \ll N$ errors, even if we just want
to store one pattern. So, if we insist that a neural network should
also exhibit some associative abilities (and this has always been a
central argument for the use of neural networks), we have to take the
graph topology into account.

This topology is encoded in the so called adjacency matrix $A$ of $G$.
Here, $A=(a_{ij})$ and $a_{ij}=1$, if $e_{i,j}\in E$ and $a_{ij}=0$
otherwise. If $G$ is sufficiently regular, the connectivity of $G$
(which played an important role in the above counterexample) can be
characterized in terms of the spectral gap.
To define it, let $\lambda_1 \ge\lambda_2 \ge\cdots\ge\lambda_N$ be
the (necessarily real) eigenvalues of $A$ in decreasing order. Define
$\kappa$ to be the second largest modulus of the eigenvalues, that is,
\[
\kappa:= \max_{i\ge2} |\lambda_i|=\max\bigl\{
\lambda_2, | \lambda _N|\bigr\}.
\]
Then the spectral gap is the difference between the largest eigenvalue
and $\kappa$, that is, $\lambda_1- \kappa$.
However, also the degrees of the vertices are important. Hence, let
$d_i = \sum_j a_{ij}$ be the degree of vertex $i$. We will denote by
\[
\delta:= \min_i d_i \quad\mbox{and}\quad m:= \max
_i d_i
\]
the minimum and maximum degree of $G$, respectively.

In this paper, we will concentrate on the parallel dynamics, which is
easier to handle when we iterate the dynamics.

%s3 #&#
\section{Results}\label{sec3}
We will now state the main result of the present paper.

In order to formulate it, let us
define the usual Hamming distance on the space of configurations~$\Sigma_N$,
\[
d_H\bigl(\sigma, \sigma'\bigr) = \tfrac{1}2
\bigl[ N - \bigl(\sigma, \sigma'\bigr)\bigr],
\]
where $(\sigma, \sigma')$ is the standard inner product of $\sigma$ and
$\sigma'$.
In other words, $d_H$ counts the number
of indices where $\sigma$ and $\sigma'$ disagree.
For any $\sigma\in\Sigma_N$ and $\varrho\in[0,1]$, let $\S(\sigma,
\varrho N)$
the sphere of radius $\varrho N$ centered at $\sigma$, that is,
\[
\S(\sigma, \varrho N) = \bigl\{ \sigma' \dvt d_H\bigl(
\sigma, \sigma'\bigr) = [\varrho N ]\bigr\},
\]
where $[\varrho N]$ denotes the integer part of $\varrho N$.

For the rest of the paper, we will suppose that the following
hypothesis is true:

%$(H_1)$ There exists $C_1 \in]0,1[,$ such that $e(J,V) - 2 e(J, I)
%\ge C_1 \lambda_1 |J| $ for all $J$ and $I$ such that $|I| \le
%\varrho_s N$.

\begin{enumerate}[(H1)]
\item[(H1)] There exists $c_1\in\,]0,1[$, such that $\delta> c_1 \lambda_1$
(recall that $\delta$ is the minimum degree of the graph $G$, and
$\lambda_1$ is the largest eigenvalue of its adjacency matrix).
\end{enumerate}

%re3.1 #&#
\begin{rem}
Condition (H1) seems to be new.
To understand it, recall that for a regular graph with degree $d$ the
largest eigenvalue of $A$ equals $d$ and so does its minimum degree
$\delta$. Condition~(H1) can thus be interpreted as the requirement
that $G$ is sufficiently regular. Indeed, it turns out that, for
example, for a Erd\"os--R\'enyi graph $G(N,p)$ is fulfilled, if and
only if $p \gg\frac{\log N}{N}$, that is, when the graph is fully
connected. Hence, for Erd\"os--R\'enyi graphs condition (H1) rules out
the sparse case, when the graph is not only disconnected asymptotically
almost surely, but also very irregular, in the sense that the degree
distribution is a Poisson distribution and the relative fluctuations of
the degrees are large. Moreover, it will turn out that also certain
power law graphs satisfy condition (H1).
\end{rem}

We will need a second condition that keeps track on how well the graph
is connected.

\begin{enumerate}[(H2)]
\item[(H2)] We say that a graph satisfies (H2), if the following relation
holds between the largest eigenvalue $\lambda_1$ of the adjacency
matrix and the modulus of its second largest eigenvalue~$\kappa$:
%
%e3.1 #&#
\begin{equation}
\label{2ndcond} {\lambda_1}\ge c \log(N) {\kappa}
\end{equation}
for some $c>0$ large enough.
\end{enumerate}

%
%re3.2 #&#
\begin{rem}
Roughly speaking, condition (\ref{2ndcond}) reveals connectivity
properties of the underlying graph. Clearly, it holds for the complete
graph $K_N$, where $\lambda_1=N-1$ and all the other eigenvalues are
equal to $-1$. Also, as pointed out below, condition (\ref{2ndcond}) is
fulfilled for an Erd\"os--R\'enyi random graph, if $p$ is large enough,
since the spectral gap, that is, the difference between the largest and
the second largest modulus of an eigenvalue is of order $Np(1-1/\sqrt{Np})$.

To understand, that indeed (\ref{2ndcond}) can be interpreted as a
measure for the connectivity of the graph, assume for a moment that the
graph were $d$-regular. Then $\lambda_1=d$. If the graph is disconnected,
there is (at least) one more eigenvalue equal to one, and hence (\ref
{2ndcond}) cannot hold. More generally, for a regular graph, the
spectrum of the adjacency matrix can be computed from the spectrum of
the Laplacian. On the other hand, the spectral gap of the Laplacian can
be estimated by Poincar\'e or Cheeger type inequalities (see \cite{diaconisstroock}), which roughly state that the spectral gap of the
Laplacian is small, if there are vertex sets of large volume, but small
surface, or if the graph has small bottlenecks. Both quantities are a
measure for how well the graph is connected.
\end{rem}

Under the above conditions, we will prove, that we can store a number
$M$ of patterns depending on $\lambda_1$ and the spectral gap of $A$ --
even in the sense that the dynamics $T$ repairs a corrupted input.
Mathematically speaking, we show the following.

%th3.3 #&#
\begin{theorem}\label{synchronous}
With the notation introduced in Section~\ref{sec2}, if \textup{(H1)} and \textup{(H2)} are satisfied,
then there exists $\alpha_c>0$ and $\varrho_1\in\,]0,1/2[$
such that if
\[
M= \alpha\frac{\lambda^2_1}{m \log N}- \frac{\kappa\lambda_1}{m} ,
\]
for some $\alpha<\alpha_c$, then that for all $\varrho\in\,]0,
\varrho
_1]$ we obtain
\[
P\bigl[\forall\mu=1,\ldots,M, \forall x % \in\S(\xi^\mu,\varrho N)
\ \mathit{s.t.}\ d_H
\bigl(x,\xi^\mu\bigr) \leq\varrho N \dvt T^k(x)=
\xi^\mu\bigr] \rightarrow 1\qquad \mbox{as } N \rightarrow\infty,
\]
for any $k \ge{C} (\max\{\log\log N, \frac{\log(N)}{\log
({\lambda_1}/{(\kappa\log(N))})}\})$ for a sufficiently large constant $C$.

Here, $T^k$ is defined as the $k$th iterate of the map $T$.
\end{theorem}

In other words, Theorem \ref{synchronous} states that we are able to
store the given number of patterns
in such a way that a number of errors that is proportional to $N$ can
be repaired by a modest (at most $\mathcal{O}(\log N)$) number of
iterations of the retrieval dynamics. The number of patterns depends on
the largest eigenvalue and the spectral gap of the adjacency matrix and
is larger for large spectral gaps.

Before advancing to the proof, we will apply this result to some
classical models of random and non-random graphs.

%co3.4 #&#
\begin{corollary}
If $G=K_N$, that is, in the case of the classical Hopfield model, the
storage capacity in the sense of Theorem \ref{synchronous} is
$M=\alpha\frac{N}{\log N}$ for some constant $\alpha$. The number of
steps needed to repair a corrupted input is of order $\mathcal{O}(\log
\log N)$.
\end{corollary}

\begin{pf}
The complete graph is regular, hence condition (H1) is satisfied.
From Theorem \ref{synchronous}, we obtain the numerical values for $M$
and the number of steps by observing that in the case of the complete
graph the eigenvalues of $A$ are $N-1$ and $-1$ (the latter being an
$N-1$-fold eigenvalue).
\end{pf}

%re3.5 #&#
\begin{rem}
It should be remarked that similar results were obtained by Komlos and
Paturi~\cite{KomlosPaturi1988}. In \cite{Komlos1993}, even the case of
regular graphs is treated. The results of these two authors were
probably inspired by the results in \cite{MPRV},
where the maximum number of patterns that are (with high probability)
fixed points of the retrieval dynamics is determined. A similar result
to \cite{KomlosPaturi1988} for the Hopfield model on the complete graph
is due to
Burshtein \cite{Burshtein}, who shows that the capacity of the Hopfield
model obtained in \cite{MPRV} does not change, if one starts with
corrupted patterns and allows for several reconstruction steps. Also, a
bound for the number of necessary steps is given. These results are
closely related to our result, and actually Burshtein is able to
determine our $\alpha$ in the case of the Hopfield model on the
complete graph. However, while he is working only with a random
corrupted input, we consider a worst case scenario since we require
that all vectors at distance $[\varrho N ]$ from the originally stored
pattern are attracted to this pattern by the retrieval dynamics. A
similar result for a Hopfield model with $q>2$ different states was
proven in \cite{LoweVermetqpotts}.

These results are to be contrasted to the findings in \cite{newman,loukianova} or \cite{talagrand}. There, one is satisfied with a
corrupted input being attracted to some point ``close'' to the stored
pattern. Naturally, the resulting capacities are larger. Also, a bound
on the number of iterations until this point is reached is not given.
\end{rem}

We mainly want to apply our results to some random architectures, that
is, $G$ will be the realization of some random graph. The most popular
model of a random graph is the Erd\"os--Renyi graph $G(N,p)$. Here, all
the possible ${ N\choose2}$ edges occur with equal probability $p=p(N)$
independently of each other.
Hopfield models on $G(N,p)$ have already been discussed in \cite{BG92,talagrand} or \cite{LV10}.

Here, we obtain the following corollary.

%co3.6 #&#
\begin{corollary}\label{GNp}
If $G$ is chosen randomly according to the model $G(N,p)$, then if $p
\ge c_0 \frac{(\log N)^2}N$ for some $c_0>0$, for a set of realizations
of $G$ the probability of which converges to one as $N \to\infty$, the
capacity (in the above sense) of the Hopfield model is $c  p N / \log
(N)$ for some constant $c>0$.
\end{corollary}

\begin{pf}
For the eigenvalues of an Erd\"os--R\'enyi graph, it is well known that
with probability converging to 1, as $N \to\infty$ (such a statement
in random graph theory is said to hold asymptotically almost surely),
$\lambda_1=(1+\mathrm{o}(1))Np$ and $\kappa\le c\sqrt{Np}$ (see, e.g.,
\cite{furedikomlos,Feige2005,Krivelevich2003} and these
facts were also used in \cite{LV10}). Moreover, we can control the
minimum and maximum degree in $G(N,p)$. Indeed, for our values of $p$
we have $m=(1+\mathrm{o}(1))Np$ and $\delta=(1+\mathrm{o}(1))Np$ asymptotically almost
surely. Surprisingly, we could not find this result in the literature,
and thus proved it in the \hyperref[app]{Appendix}.

Hence, (H1) is satisfied.
\end{pf}

%re3.7 #&#
\begin{rem}
As mentioned above, the Hopfield model on an Erd\"os--R\'enyi graph has
already been discussed in \cite{BG92,talagrand} or \cite{LV10}.
The first two of these papers treat the case of rather dense graphs,
more precisely the regime of $p \ge\operatorname{const.} \sqrt{\frac{\log N}N
}$. This regime seems to be a bit artificial, since a realization of
$G(N,p)$ is already connected, once $p$ is larger than $\frac{\log N}N
$. The regime of
$ \operatorname{const.}_1 \frac{\log N} N \le p \le\operatorname{const.}_2
\sqrt
{\frac{\log N} N}$ was analyzed in \cite{LV10}. However, in all of
these papers the notion of storage capacity is the one, where we just
require stored patterns to be close to minima of the energy function,
that is, fixed points of the retrieval dynamics. As motivated above,
this notion is unable to reflect the different reconstruction abilities
for various network architectures. Corollary \ref{GNp} deals with the
notion of storage capacity introduced in Section~\ref{sec2}; one might naturally
wonder, whether the restriction $p \ge c_0 \frac{(\log N)^2}N$ could be
weakened or whether this is the optimal condition, when we consider
this notion of storage capacity. However, by now we do not have an
answer to this question, especially since the reverse bound on the
storage capacity is usually much harder to obtain.
\end{rem}

The next example is one of the central results of the present paper: We
analyze the Hopfield model on an architecture that comes closer to the
models used in neuroscience, the so-called power law graphs.
To introduce it, let us give a general construction of random graph
models, which is standard in graph theory (see, e.g., \cite{ChungLu2002AC}
or \cite{Chung-Lu02}) and
nowadays referred to as the Inhomogeneous Random
Graph (see, e.g., the very recommendable lecture notes \cite{vdH}).
To this end,
let $i_0$ and $N$ positive integers and $L=\{i_0, i_0+1, i_0+N-1\}$.
For a sequence $w=(w_i)_{i\in L}$, we consider random graphs $G(w)$ in
which edges are assigned independently to each pair of vertices $(i,j)$
with probability
\[
p_{ij}=\varrho w_i w_j,
\]
where $\varrho=1/\sum_{k\in L} w_k$. We assume that
\[
\max_i w_i^2 < \sum
_{k\in L} w_k
\]
so that $p_{ij} \leq1$ for all $i$ and $j$.
It is easy to see that the expected degree of $i$ is $w_i$. This allows
for a very general construction of random graphs. Note in particular
that for $w_i = pN$ for all $i= 1, \ldots, N$, one recovers the Erd\"
os--R\'enyi graph.

For notational convenience, let
\[
{ d}= \sum_{i\in L} w_i /N
\]
be the expected average degree, ${\overline m}$ the expected maximum
degree and
\[
\tilde{d}= \sum_{i\in L} w_i^2
\Big/ \sum_{i\in L} w_i
\]
be the so-called second-order average degree of the graph $G(w)$. From
these definitions, the advantage of this kind of construction of a
random graph becomes transparent: We are able to construct random
graphs, with expected degrees that are up to our own choice.

We now turn to a subclass of random graphs that have recently become
very popular, power law graphs \cite{durrett_graphs}.
Power law random graphs are random graphs in which the number of
vertices of degree $k$ is proportional to $1/k^\beta$ for some fixed
exponent $\beta$. It has been realized that this ``power law''-behavior
is prevalent in realistic graphs arising in various areas. Graphs with
power law degree distribution are ubiquitously encountered, for
example, in
the internet, the telecommunications graphs, the neural networks and
many biological applications \cite{jeongetal,Schreiber,ZhouLipowsky}.
The common feature of such networks is that they are large, have small
diameter, but have small average degree. This behavior can be achieved
by hubs, a few vertices with a much larger degree than others. A possible
choice would be a power law graph, where the degrees obey a power law
distribution. Keeping in mind that the $G(w)$ model allows to build a
graph model with a given expected degree sequence, it is plausible that
this model can be used to model the networks of the given examples.
Indeed, using the $G(w)$ model, we can build random power law graphs in
the following way. Given a power law exponent $\beta$, a maximum
expected degree ${\overline m}$, and an average degree $d$, we take
$  w_i=c i^{-{1}/{(\beta-1)}}$ for each $i\in\{i_0,
\ldots
, i_0+1, i_0+N-1\}$, with
\[
c= \frac{\beta-2}{\beta-1} d N^{{1}/{(\beta-1)}}
\]
and
\[
i_0= N \biggl(\frac{d(\beta-2)}{{\overline m}(\beta-1)} \biggr)^{\beta-1}.
\]

For such power law graphs, we obtain the following.

%co3.8 #&#
\begin{corollary}\label{powerlawcor}
If $G$ is chosen randomly according to a power law graph with $\beta
>3$, then if
\[
{\overline m}\gg d > c \sqrt{{\overline m}}\bigl(\log(N)\bigr)^{3/2}
\]
or
\[
{\overline m} \gg d> c \sqrt{{\overline m}} \bigl(\log(N)\bigr) \quad\mbox{and}\quad {
\overline m}\gg(\log N)^4,\vadjust{\goodbreak}
\]
for some constant $c>0$ for a set of realizations of $G$ the
probability of which converges to one as $N \to\infty$, the capacity
(in the above sense) of the Hopfield model is $C(\beta) \frac
{d^2}{{\overline m} \log(N)}$ for a constant $C$ that only depends on
$\beta$.
\end{corollary}

%re3.9 #&#
\begin{rem}
\begin{itemize}
\item
One might indeed wonder, whether the restriction of $\beta>3$ is an
artefact of our proof below or whether there is some intrinsic reason,
why storing become much more difficult for \mbox{$\beta<3$}. A recent paper
by Jacob and M\"orters
\cite{JacobMorters} may shed some light on this question. There a
spatial preferential attachment graph is constructed (for details, see
the construction in \cite{JacobMorters}). It turns out that the graphs
have powerlaw behavior for the degree distribution. The parameter
$\beta
$ depends on the parameters of the model. For their model, the authors
are able to show that for $\beta>3$ the models exhibits clustering,
that is, many triangles occur, while for
$\beta<3$ there is no clustering. On the other hand, the storage of
patterns in a Hopfield is basically a collective phenomenon for which a
strong interaction of the neurons is necessary. Clustering is a measure
for such a strong interaction.
\item The second condition in Corollary \ref{powerlawcor} basically
states that we assume that there are so-called hubs, that is, vertices
with a much larger degree than the average one, but that the graph may
not be too irregular, for example, for a star graph (one vertex
connected to all other vertices that are not connected otherwise), this
condition would be violated, and indeed we would not be able to repair
corrupted patterns on such a graph.
\end{itemize}
\end{rem}

\begin{pf*}{Proof of Corollary \ref{powerlawcor}}
By definition, if $d \ll{\overline m}$, then the minimum expected
degree $w_{\mathrm{min}}= c (i_0+N-1)^{-{1}/{(\beta-1)}}$ satisfies $w_{\mathrm{min}}=
\frac{\beta-2}{\beta-1} d (1+\mathrm{o}(1))$.

From \cite{Chung-Lu02}, we learn about the second-order average degree that
\[
\tilde{d}= \bigl(1+\mathrm{o}(1)\bigr) \frac{(\beta-2)^2}{(\beta-1)(\beta-3)} d,
\]
if $\beta>3$.

On the other hand, Chung and Radcliffe prove in \cite{Chung-Radcliffe11} the following: if the maximum expected degree
${\overline m}$ satisfies ${\overline m}> \frac{8}9 \log(\sqrt{2} N)$,
then with probability at least $1- \frac{1}N$, we have
\[
\lambda_1\bigl(G(w)\bigr) =\bigl(1+\mathrm{o}(1)\bigr) \tilde{d}\quad
\mbox{and}\quad\kappa\bigl(G(w)\bigr)\le \sqrt{8 {\overline m} \log(\sqrt{2}N)}.
\]

We will now use the following exponential bound due to Chung and Lu.
As shown by these authors in \cite{ChungLu2002AC}, we have the
following estimate, using Chernoff inequalities: for all $c>0$, there
exist two constants $c_0, c_1>0$ such that
%
%e3.2 #&#
\begin{equation}
\label{ChungLubound} P\bigl[ \exists i\in L \dvt | d_i - w_i | > c
w_i \bigr] \le\sum_{i\in L} \exp(-
c_0 w_i) \le\exp( -c_1 d + \log N),
\end{equation}
since $w_{\mathrm{min}} = \mathcal{O}(d)$.
Applying this with, for example, $c=1/2$, we see (applying the
Borel--Cantelli lemma) that for almost all realizations of the random
graphs, we have that for all $i\in L$,
\[
d_i > \frac{1}2 w_i \ge\frac{1}2
w_{\mathrm{min}} = \frac{1}2\frac{\beta
-2}{\beta
-1} d \bigl(1+\mathrm{o}(1)
\bigr)= \frac{1}2\frac{\beta-3}{\beta-2} \lambda_1 \bigl(1+
\mathrm{o}(1)\bigr),
\]
and thus
\[
\delta> \frac{1}2\frac{\beta-3}{\beta-2} \bigl(1+\mathrm{o}(1)\bigr)
\lambda_1,
\]
which is (H1).

To apply Theorem \ref{synchronous}, we also need to compare
${\overline
m}$ to the maximum degree $m$ of a graph $G$, chosen randomly according
to a power law graph with $\beta>3$.
We again use \eqref{ChungLubound}.
%As proven in \cite{ChungLu2002AC}, we have the following estimate,
%using Chernoff inequalities~: for all $c>0$, there exist two constants
%$c_0, c_1>0$ such that
%$$P[ \exists i\in L: | d_i - w_i | > c w_i ] \le\sum_{i\in L} \exp(-
%c_0 w_i) \le\exp( -c_1 d + \log N),$$
%since $w_{\mathrm{min}} = \mathcal{O}(d)$.

Under our assumption that $d\gg\log N$, we deduce from this estimate
that $m= C  {\overline m}(1+\mathrm{o}(1))$, for some $C>0$, and we
finally obtain that the capacity of the Hopfield model on power law
graphs (for a sequence of sets of graphs with probability converging to
one) is at least
\[
\operatorname{const.} \frac{\lambda^2_1}{m \log(N)}- \frac{\kappa\lambda
_1}{m}= C(\beta)
\frac{d^2}{{\overline m} \log(N)},
\]
if, $\beta>3$, and $\kappa<c_2 \frac{ \lambda_1}{\log(N)}$ for some
$c_2>0$ small enough. This is true in particular, if
\[
\sqrt{8 {\overline m} \log(\sqrt{2}N)} < c_3 \frac{d}{ \log(N)}
\]
for some $c_3$ small enough, that is,
\[
d> c \sqrt{{\overline m}} \bigl(\log(N)\bigr)^{3/2}.
\]

In fact, this condition on $d$ can be slightly weakened, if we consider
the slightly stronger condition on the maximum expected degree:
${\overline m}\gg(\log N)^4$. Indeed, in a recent paper \cite{LuPeng}, Lu and Peng prove that under this condition on ${\overline
m}$, we have
\[
\lambda_1\bigl(G(w)\bigr)= \bigl(1+\mathrm{o}(1)\bigr)\tilde{d}\quad
\mbox{and}\quad \kappa \bigl(G(w)\bigr)\le2 \sqrt{ {\overline m}} \bigl(1+\mathrm{o}(1)
\bigr),
\]
a.s., if $\tilde{d} \gg\sqrt{ {\overline m}}$. Finally, we get as
previously a capacity of order
\[
C(\beta) \frac{d^2}{{\overline m} \log(N)},
\]
if ${\overline m} \gg d> c \sqrt{{\overline m}} (\log(N))$ and
${\overline m}\gg(\log N)^4$.
\end{pf*}

%s4 #&#
\section{Technical preparations on random graphs}\label{sec4}
We first present the results we will use in the proof of our theorem.
Let $G$ be a simple graph with $N$ vertices and $l$ edges. Recall that
for such a graph
\[
\lambda_1\ge\cdots\ge\lambda_N
\]
are the (real) eigenvalues of its adjacency matrix and $\kappa= \max\{
\lambda_2, | \lambda_N|\}$.

We begin with an estimate of the moment generating function of a sum of
i.i.d. random variables, related to $G$.
We assign i.i.d. random variables $X_i$ to the vertices of $G$, taking
values $\pm1$ with equal probability. Let us define the ``quadratic
form'' over $G$
\[
S= \sum_{\{i,j\} \in E} X_i X_j.
\]

The following theorem due to Komlos and Paturi \cite{Komlos1993} gives
an upper bound on the moment generating function of $S$, which appears
naturally when we use an exponential Markov inequality for an upper bound.

%th4.1 #&#
\begin{theorem}[(\cite{Komlos1993})]\label{Ee(tS)}
The moment generating function of $S$ can be bounded as
\[
E \bigl[ \mathrm{e}^{-tS}\bigr]\leq E\bigl[\mathrm{e}^{tS}
\bigr] \le\exp\biggl(\frac{l t^2}{2(1-\lambda_1 t)}\biggr),
\]
for $0\leq t < 1/\lambda_1$.
\end{theorem}

%re4.2 #&#
\begin{rem}
The attentive reader may wonder, whether the above theorem is really
difficult to prove, as the random variables $X_i X_j$ are Bernoulli
random variables. However, note that they are not independent, which is
the basic difficulty in this estimate.
\end{rem}

Not unexpectedly, a bound on the moment generating function implies a
concentration of measure result.

%co4.3 #&#
\begin{corollary}\label{ExpMarkov} For any $y>0$, we have
\[
P[S> y] \leq\exp\biggl(-\frac{y^2}{2(l+ \lambda_1 y)}\biggr).
\]
%
%and
%$$  P[S> y] \leq\exp(-\frac{1-A}{\lambda_1}{(y+ \frac l{2
%\lambda_1}(1- \frac1A}))),$$
%where $  A= \sqrt{\frac l{2\lambda_1 y + l}}.$
\end{corollary}

\begin{pf}
Apply the exponential Markov inequality together with Theorem \ref
{Ee(tS)} to see that
\[
P[S> y] \leq \mathrm{e}^{-t y} E\bigl[ \mathrm{e}^{t S}\bigr]
\leq\exp\biggl(-ty +\frac{l
t^2}{2(1-\lambda_1 t)}\biggr),
\]
for $0\leq t < 1/\lambda_1$.
The desired estimate is obtained by the choice of $t=\frac{y}{l+ \lambda
_1 y}$ which is smaller than $1/\lambda_1$.
% To get the second estimate, we choose the value $t= \frac1 {
%\lambda_1} (1-A)$, which minimizes the function $f(t) =-ty +\frac{l
%t^2}{2(1-\lambda_1 t)}$.
\end{pf}

As we will apply this result for subgraphs in the proof of our main
result, we need also an estimate of the largest eigenvalue $\lambda
_1(H)$ of particular subgraphs $H$ of $G$. To this end, we will quote
another result by Komlos and Paturi \cite{Komlos1993}.

%le4.4 #&#
\begin{lemma}[(\cite{Komlos1993})]\label{subgraph}
Let $G$ be a simple graph with $N$ vertices. If $I$ and $J$ are two
subsets of the vertex set of $G$ with $|I|=\varrho N$ and $|J|=\varrho'
N$, where $\varrho, \varrho' \in(0,1)$,
the number of %directed
edges $e(J;I)$ going from $J$ to $I$ is at most
\[
e(J,I) \le\bigl[ \varrho\varrho' \lambda_1(G) + \sqrt{
\varrho\varrho' } \kappa(G) \bigr] N.
\]

Moreover, the largest eigenvalue (of the adjacency matrix) of the
graph $H$ determined by the %undirected
edges from $I$ to $J$ is bounded as
\[
\lambda_1(H) \le2\bigl[\sqrt{\varrho\varrho'}
\lambda_1(G) + \bigl(1-\sqrt {\varrho\varrho'} \bigr)
\kappa(G)\bigr].
\]
\end{lemma}

The proof of this lemma basically involves estimating quadratic forms
by their eigenvalues together with Cauchy's interlacing theorem for
eigenvalues of matrices. However, it is not trivial (see the proof in
\cite{Komlos1993}).

%s5 #&#
\section{Proof of the main result}\label{sec5}

We are now ready to begin with the proof of Theorem \ref{synchronous}.
%\subsection{Proof of the Theorem \ref{synchronous}:}
We first present an important lemma that determines the behavior of the
system for one step of the synchronous dynamics, more precisely it
controls, how many errors are corrected by one step of the dynamics.

%le5.1 #&#
\begin{lemma}\label{mainlemma}
%There exists $\varrho_1, c_1, c_2 >0$, such that the following holds
%with probability $1-\varepsilon_N$, where $\varepsilon_N>0$ and $
%\varepsilon_N\rightarrow0$ as $N\rightarrow+\infty$:
%Let $ \varrho_0=\exp(-c_2\frac{\lambda_1}{\mu+M})$ and
%$x$ be a vector at a distance of $\varrho N$ from a fundamental memory
%$\xi^\mu$, where $\varrho_0\leq\varrho\leq\varrho_1$. Let $T(x)$ be
%the state after one step of the synchronous dynamics. Then $d_H(T(x),
%\xi^\mu) \le\min\{ f(\varrho), \varrho\} N$, where $  f(
%\varrho)= c_1\max\{\varrho h(\varrho),\frac{\mu}{\lambda_1} h(
%\varrho), \varrho (\frac{M}{\lambda_1} \frac{\mu}{\lambda_1} \log(
%\frac1\varrho))^{2/3}, \exp(- c_2 \frac{\lambda_1}{ \mu+ M}) \}.$
Recall that $m$ denotes the maximum degree of the random graph $G$ in
question and let
\[
\varrho_0=\exp\biggl(-c_2\frac{\lambda_1}{\kappa+{Mm}/{\lambda_1}}
\biggr),
\]
for some constant $c_2>0$.
If $M\le c \lambda_1$ for some constant $c>0$, there exists $\varrho
_1\in(0,\frac{1}2)$ and a constant $c_1>0$, such that for all $\varrho
\in[\varrho_0,\varrho_1]$ we have
\[
P\bigl[\forall\mu\in\{1,\ldots,M\}, \forall x\in\S\bigl(\xi^\mu,
\varrho N\bigr) \dvt d_H\bigl(T(x), \xi^\mu\bigr)\leq f(
\varrho) N\bigr] \ge1-\varepsilon_N,
\]
where
\[
f(\varrho)= \max\biggl\{c_1 \varrho \biggl(
\frac{\kappa}{\lambda_1}\biggr)^2, c_1 \varrho h(
\varrho),c_1 \frac{\kappa}{\lambda_1} h(\varrho),c_1 \varrho
\biggl(\frac{M\kappa}{(\lambda_1)^2} \log\biggl(\frac{1}\varrho\biggr)
\biggr)^{2/3}, \varrho _0\biggr\}\leq\varrho,
\]
$\varepsilon_N \ge0$, $\varepsilon_N\rightarrow0$ as $N\rightarrow
+\infty$ and
\[
h(\varrho)= -\varrho\log\varrho- (1-\varrho) \log(1-\varrho)
\]
is the entropy function.
\end{lemma}

%\underline{Remark:} Note that $\varrho_1$ can it will depend on $M$ to
%ensure.
%$f(\varrho)\le\varrho$ \rot{is this a problem?}

%
\begin{pf} %We want to show that
%$$P[\forall\mu\in\{1,\ldots,M\}, \forall x\in\S(\xi^\mu, \varrho N)
%: d_H(T(x), \xi^\mu)\leq f(\varrho)N] \ge1-\varepsilon_N.$$
This lemma is of central importance for our main result. However, its
proof is rather technical. Let us therefore first describe its basic idea.

To this end, recall that it suffices to prove that
%
%e5.1 #&#
\begin{equation}
\label{prob1} \sum_{\mu=1}^M P\bigl[\exists
x\in\S\bigl(\xi^\mu, \varrho N\bigr) \dvt d_H\bigl(T(x),
\xi^\mu \bigr)> f(\varrho)N\bigr] \le\varepsilon_N.
\end{equation}
To simplify notation, we can assume that the fundamental memory in
question is $\xi^1$.

Now assume that we start with a corrupted input (i.e., a corrupted pattern)
$x\in\{-1,1\}^N$ such that $d_H(\xi^1,x)= \varrho N$.
Let $I$ be the set of coordinates in which $x$ and $\xi^1$ differ.
Let $T(x)$ be the vector resulting after one step of the parallel
dynamics, and $J$ be the set of coordinates in which $T(x)$ and $\xi
^1$ differ.
Now define the weight matrix $W$ as $W=(w_{ij})$ and
\[
w_{ij}=a_{ij}\sum_{\nu=1}^M
\xi^\nu_i \xi^\nu_j.
\]
Then, since $\xi^1$ is not properly reconstructed for the coordinates
$j \in J$,
for all $j\in J$, we have $ \xi_j^1 (W x)_j \leq0$, which implies
$ \sum_{j\in J} \xi_j^1 (W x)_j \leq0$.

The idea is now to analyze the contributions to $\sum_{j\in J} \xi_j^1
(W x)_j$. Similar to what we said in the analysis of the dynamics $T_i$
in Section~\ref{sec2}, there is a ``signal term''
stemming from the closeness of $x$ to $\xi^1$ and there are noise terms
from the influence of the other patterns. We will first show that the
signal term grows at least linearly in $|J|$. On the other hand, we
are also able to give an upper bound on the influence of the random
noise terms that are also controlled by the size of $I$ and $J$. While
all these computations are relatively straight forward in the Hopfield
model on the complete graph, the estimates become much more involved on
a general graph. The key observation is that we are able to control the
probability to find sets $I$ and $J$ with the above properties with the
help of the spectrum of the adjacency matrix (using the results of the
previous section). Technically to this end, we have to split up the
noise terms according to where the vertices $i$ in $\sum_{j\in J} \sum_I \xi_j^1 a_{ij}\xi^\mu_i \xi^\mu_j$
come from.
The bottom line is, that if $|J|$ is too large, the probability to find
sets $I$, with $|I|=\varrho N$ and $J$ (such that $\xi^1$ is not
reconstructed correctly on $J$ when starting with an $x$ differing from
$\xi^1$ in the coordinates $I$) converges to 0 -- even when being
multiplied by the number of patterns $M$, if $M$ is of the given size
(cf. equation \eqref{centest} below).

Let us now carry out this idea.

For later use, set
\[
S^\mu(J,I) = \sum_{j\in J}\sum
_{k=1}^N a_{jk} \xi_j^1
\xi_j^\mu \xi _k^\mu
x_k
\]
and
\[
S(J,I)= \sum_{\mu=1}^M S^\mu(J,I)=:
\sum_{j\in J} \xi_j^1 (W
x)_j.
\]

Observe that, if the patterns are chosen i.i.d. with i.i.d. coordinates
their typical distance is $N/2 \pm\operatorname{const.} \sqrt N$. This in
turn implies that,
if $\varrho<1/2$ and $d_H(x, \xi^1)=\varrho N$, then $x$ tends to be
closer to $\xi^1$ than to any other pattern, and $S^1(J,I)$ will be the
dominating term in $S(J,I)$.
We will first give a lower bound for $S^1(J,I)$. We can rewrite
$S^1(J,I)$ as
%
%$\begin{array}{ll}
%T^1(J,I)& = \sum_{j\in J}\sum_{k=1}^N A_{jk} \xi_k^1 x_k\\&
% = \sum_{j\in J} (e(j,\bar I) - e(j,I) )\\&= e(J,V) - 2
%e(J,I)\\
%\end{array}$
%
\[
S^1(J,I) = \sum_{j\in J}\sum
_{k=1}^N a_{jk} \xi_k^1
x_k = \sum_{j\in J} \bigl(e(j,\bar I) -
e(j,I) \bigr)= e(J,V) - 2 e(J,I),
\]
where again we use the notation $e(J,I)$ and $e(j,I)$, to denote the
number edges going from the set $J$ to the set $I$, or, respectively,
from the vertex $j$ to the set $I$. Moreover, $\bar I$ denotes the
complement of the set $I$ in $V$.

Under the assumption of hypothesis (H1) and with the help of Lemma
\ref{subgraph}, we have for all $I$ and $J$,
\begin{eqnarray*}
S^1(J,I)& \ge& c_1
\lambda_1 |J| - 2 \biggl(|I| |J| \frac
{\lambda
_1}N + \sqrt{|I| |J|}
\kappa\biggr)
\\
& = &\lambda_1 |J| \biggl( c_1 - 2 \varrho- 2 \sqrt{
\frac
{\varrho
}{\varrho'}} \frac{\kappa}{\lambda_1}\biggr),
\end{eqnarray*}
where $\varrho' = \frac{|J|}N$.
If we assume that $ \varrho'\ge c_2 \varrho (\frac
{\kappa
}{\lambda_1})^2$ for some $c_2>0$ large enough, and $\varrho<\varrho_1$
for some $\varrho_1\in(0,1/2)$ small enough, we get
%
%e5.2 #&#
\begin{equation}
\label{lowerboundS1} S^1(J,I)\ge C_1 \lambda_1 |J|,
\end{equation}
for some constant $C_1\in(0,1)$.

For $\mu\ge2$, we compute
\begin{eqnarray*}
S^\mu(J,I)&=& \sum
_{(j,k)\in E(J,\bar I)} u^\mu_ju_k^\mu-
\sum_{(j,k)\in E(J,I)} u^\mu_ju_k^\mu
\\
&=& \sum_{(j,k)\in E(J,V)} u^\mu_ju_k^\mu-2
\sum_{(j,k)\in
E(J,I)} u^\mu_ju_k^\mu,
\end{eqnarray*}
where $u_i^\mu=\xi^1_i\xi_i^\mu$, for all $i=1,\ldots,N$ and $\mu
=1,\ldots, M$.
To apply the results for the moment generating function of quadratic
forms introduced in Theorem \ref{Ee(tS)} and Corollary \ref{ExpMarkov},
we need to rewrite these sums over ordered pairs of vertices as sums
over unordered pairs. We have
\[
E(J,V) = E(J,J) + E(J, \bar J) = 2 E\{J,J\} + E\{J, \bar J\}= E\{J,V\} + E\{J, J
\},
\]
where for $K,L \subset V$ $E(K,L)$ is the edges set of the directed
graph between the sets $K$ and $L$ induced by our original graph.
Likewise, $E\{K,L\}$ denotes the corresponding set of undirected edges.
In the same way, we obtain
\begin{eqnarray*} E(J,I)&=& E(J\cap\bar I, J\cap I) + E(J\cap I,J\cap
I) + E(J, I\cap \bar J)
\\
&= &E\{J\cap\bar I, J\cap I\} + 2 E\{J\cap I,J\cap I\} + E\{J, I\cap \bar J\}
\\
&=& E\{J, I\} + E\{J\cap I,J\cap I\}.
\end{eqnarray*}
Eventually,
\[
E(J,V) - 2 E(J,I)= E\{J,V\} - 2 E\{J, I\} + E\{J, J\} - 2 E\{J\cap I,J\cap I\}.
\]

We want to prove that for $\varrho'$ larger than $f(\varrho)$
we have that
%
%e5.3 #&#
\begin{equation}
\label{centest} M P\bigl[\exists I, |I|= \varrho N,\exists J, |J|=
\varrho' N, S(J,I) <0\bigr] \longrightarrow0,
\end{equation}
as $N\rightarrow+\infty$.

To this end, set
\begin{eqnarray*}
S_1^\mu(J)&=& \sum
_{(j,k)\in E\{J,V\}} u^\mu_ju_k^\mu,\qquad
S_2^\mu (J,I)= \sum_{(j,k)\in E\{J,I\}}
u^\mu_ju_k^\mu,
\\
S_3^\mu(J)&=& \sum_{(j,k)\in E\{J,J\}}
u^\mu_ju_k^\mu\quad \mbox{and}\quad S_4^\mu(J,I)= \sum
_{(j,k)\in E\{J\cap I,J\cap I\}} u^\mu_ju_k^\mu.
\end{eqnarray*}
Then
\[
S(J,I) = S^1(J,I)+ \sum_{\mu=2}^M
S^\mu_1(J) -2 \sum_{\mu=2}^M
S^\mu _2(J,I)+ \sum_{\mu=2}^M
S^\mu_3(J)- \sum_{\mu=2}^M
S^\mu_4(J,I).
\]
Let $\gamma_1, \gamma_2, \gamma_3, \gamma_4 \ge0$, such that
$\gamma
_1+2\gamma_2+\gamma_3+\gamma_4=1$.

We will consider the four sums separately.
First, using (\ref{lowerboundS1}), we have
\begin{eqnarray*}
%{ll}
&&P\Biggl[\exists I, |I|= \varrho N,\exists J, |J|=
\varrho' N, \sum_{\mu
=2}^M
S^\mu_1(J) <-\gamma_1 S^1(J,I)
\Biggr]
\\
&&\quad \leq%
\sum_{J \dvt |J|= \varrho' N} P\Biggl[\sum
_{\mu=2}^M S^\mu_1(J) <-
\gamma_1 C_1 \lambda_1 |J|\Biggr].
\end{eqnarray*}

Given the vector $\xi^1=(\xi_i^1)_{i=1,\ldots, N}$, the random
variables $(u_i^\mu)_{i=1,\ldots, N}^{\mu=2,\ldots, M}$ are
conditionally independent and uniformly distributed on $\{-1,+1\}$. As
the estimates we will get for the conditional probabilities and the
moment generating function will not depend on the choice of $\xi^1$,
they will be true also for the unconditional probabilities.

Given the vector $\xi^1$, the random variables $S^\mu_1(J), \mu
=2,\ldots
, M$, are independent. Similar to the estimate of Corollary \ref
{ExpMarkov}, we obtain
\[
P\Biggl[\sum_{\mu=2}^M S^\mu_1(J)
<-\gamma_1 C_1 \lambda_1 |J|\Biggr] \leq \exp
\biggl(-\frac{1}2\frac{\gamma_1 C_1 \lambda_1 |J|}{\lambda_J+
{M
e\{J,V\}}/{(\gamma_1 C_1 \lambda_1 |J|)}} \biggr),
\]
where $\lambda_J=\lambda_1(E\{J,V\})$ is the largest eigenvalue of the
graph determined by the undirected edges in $E\{J,V\}$.
Using Lemma \ref{subgraph}, we have
\[
\lambda_J \le2\bigl[\sqrt{\varrho'}
\lambda_1 + \bigl(1-\sqrt{\varrho'} \bigr) \kappa\bigr],
\]
and
%$e\{J,V\}\leq e(J,V) \le(\varrho' \lambda_1 +\sqrt{\varrho'}\mu) N$.
moreover, $e\{J,V\}\leq e(J,V) \le m |J|$ is trivially true.
We deduce that
\[
P\Biggl[\sum_{\mu=2}^M S^\mu_1(J)
<-\gamma_1 C_1 \lambda_1 |J|\Biggr] \leq \exp
\biggl(-\frac{\gamma_1 C_1}2 \frac{\varrho'N}{2\sqrt{\varrho'}+2{\kappa
}/{\lambda_1}+ {Mm}/{(\gamma_1 C_1(\lambda_1)^2)}} \biggr).
\]

Now there are ${N\choose|J|}$ ways to choose the set $J$, and by
Stirling's formula
\[
\pmatrix{N
\cr
|J|}\leq\exp\bigl(h\bigl(\varrho'\bigr)N\bigr),
\]
where
\[
h(x)= -x \log x -(1-x) \log(1-x)
\]
is the entropy function introduced above.

Using
$h(\varrho') \leq-2 \varrho' \log(\varrho')$, we obtain that
\begin{eqnarray*}
%{ll}
&& \sum_{J \dvt |J|= \varrho' N} P\Biggl[\sum
_{\mu=2}^M S^\mu_1(J) <-
\gamma_1 C_1 \lambda_1 |J|\Biggr]
\\
&&\quad\le \exp \biggl(-2 \varrho'N \biggl(\frac{\gamma_1 C_1}4
\frac
{1}{2\sqrt
{\varrho'}+2{\kappa}/{\lambda_1}+ {Mm}/{(\gamma_1
C_1(\lambda
_1)^2)}} + \log\bigl(\varrho'\bigr) \biggr) \biggr).
\end{eqnarray*}
The exponent is negative, if
\[
\frac{\gamma_1 C_1}4 \frac{1}{2\sqrt{\varrho'}+2{\kappa
}/{\lambda
_1}+ {Mm}/{(\gamma_1 C_1(\lambda_1)^2)}} + \log\bigl(\varrho'\bigr)
>0,
\]
which is true if
%
%e5.4 #&#
\begin{equation}
\label{firstcond} \frac{\gamma_1 C_1}{8} \frac{1}{2\sqrt{\varrho'}} + \log\bigl(\varrho
'\bigr) >0,
\end{equation}
as well as
%
%e5.5 #&#
\begin{equation}
\label{secondcond} \frac{\gamma_1 C_1}{8}\frac{\lambda_1}{2{\kappa}+
{Mm}/{(\gamma
_1C_1\lambda_1)}}+ \log\bigl(
\varrho'\bigr) >0.
\end{equation}
This gives a first bound on $f(\varrho)$ in the sense, that if
$\varrho
'$ is so large, then we will have small probabilities to find the
corresponding sets $I$ and $J$.

Now, there exists a $\varrho_1\in(0,0.1) $, such that the first
condition \eqref{firstcond} is true if $\varrho'<\varrho_1$.
The second condition \eqref{secondcond} is true if
\[
\varrho' > \exp\biggl(- c \frac{\lambda_1}{2 \kappa+ {Mm}/{(\gamma
_1C_1\lambda_1})}\biggr),
\]
where $c= \frac{\gamma_1C_1}{8}$.
This implies that, if there exists a constant $c_2>0$ such that
\[
\varrho' \ge\varrho_0:= \exp\biggl(- c_2
\frac{\lambda_1}{ \kappa+
{Mm}/{\lambda_1}}\biggr),
\]
then \eqref{secondcond} is true.
%\rot{
%The third condition is true if
%$$\frac{\log(1/\varrho')}{\sqrt{\varrho'}} < c_3 \frac{(\lambda_1)^2}{M
%\mu},$$
%where $c_3=\frac{(\gamma_1 C_1)^2}{12}$.}

For the second term, we have
\begin{eqnarray*}
&&P\Biggl[\exists I, |I|= \varrho N,\exists J, |J|=
\varrho' N, \sum_{\mu=2}^M
S^\mu_2(J,I) >\gamma_2 S^1(J,I)
\Biggr]
\\
&&\quad\leq \sum_{I \dvt |I|= \varrho N}\sum_{J \dvt |J|= \varrho' N}
P\Biggl[\sum_{\mu=2}^M S^\mu_2(J,I)
>\gamma_2 C_1 \lambda_1 |J|\Biggr]
\\
&&\quad\leq\sum_{I \dvt |I|= \varrho N}\sum_{J \dvt |J|= \varrho' N}
\exp \biggl(-\frac{1}2\frac{\gamma_2 C_1 \lambda_1 |J|}{\lambda_{\{
J,I\}
}+ {M e\{J,I\}}/{(\gamma_2 C_1 \lambda_1 |J|)}} \biggr),
\end{eqnarray*}
where $\lambda_{\{J,I\}}=\lambda_1(E\{J,I\})$ is the largest eigenvalue
of the graph determined by the undirected edges in $E\{J,I\}$.
Using Lemma \ref{subgraph}, we get
\[
\lambda_{\{J,I\}} \le2\bigl[\sqrt{\varrho\varrho'}
\lambda_1 + \kappa\bigr]
\]
and
\[
e\{J,I\} \le\bigl(\varrho\varrho' \lambda_1 +\sqrt{
\varrho\varrho '}\kappa\bigr) N,
\]
which implies
\begin{eqnarray*}
&&P\Biggl[\sum_{\mu=2}^M S^\mu_2(J)
>\gamma_2 C_1 \lambda_1 |J|\Biggr] \\
&&\quad\leq\exp
\biggl(-\frac{\gamma_2 C_1}2 \frac{\varrho' N}{2\sqrt{\varrho\varrho
'}+2
{\kappa}/{\lambda_1}+{M\varrho}/{(\gamma_2 C_1\lambda_1)} +
({M\kappa}/{(\gamma_2 C_1(\lambda_1)^2)})\sqrt{{\varrho}/{\varrho
'}}} \biggr).
\end{eqnarray*}
There are ${N\choose|I|} {N\choose|J|} $ ways to choose the sets $I$
and $J$ and
\[
\pmatrix{N
\cr
|I|} \pmatrix{N
\cr
|J|} \leq\exp(\bigl(h(\varrho )+h\bigl(
\varrho'\bigr)N\bigr)\leq \exp\bigl(2 h(\varrho) n\bigr),
\]
as we assume that $\varrho'\le\varrho\le1/2$.
These considerations yield that
\[
P\Biggl[\exists I, |I|= \varrho N,\exists J, |J|= \varrho' N, \sum
_{\mu=2}^M S^\mu_2(J,I)
>\gamma_2 S^1(J,I)\Biggr]
\]
becomes small, once the condition
\[
\frac{\gamma_2 C_1}2 \frac{\varrho'}{2\sqrt{\varrho\varrho
'}+2
{\kappa}/{\lambda_1}+{M\varrho}/{(\gamma_2 C_1\lambda_1)} +
({M\kappa}/{(\gamma_2 C_1(\lambda_1)^2)})\sqrt{{\varrho}/{\varrho
'}}} > 2 h(\varrho),
\]
is satisfied.
This is true if
\[
\gamma_2 C_1 \frac{\varrho'}{4\sqrt{\varrho\varrho'} } > 8 h(\varrho),\qquad
%$$
%$$
 \gamma_2 C_1 \frac{\varrho'}{4 {\kappa}/{\lambda_1}} >
8 h(\varrho), \qquad%$$
%$$
 \frac{\gamma_2 C_1}2 \frac{\varrho'}{{M\varrho
}/{(\gamma
_2C_1\lambda_1)} } > 8 h(
\varrho)
\]
and
\[
\frac{\gamma_2 C_1}2\frac{\varrho'}{ ({M\kappa}/({\gamma_2
C_1(\lambda
_1)^2}))\sqrt{{\varrho}/{\varrho'}}} > 16 \varrho\log\biggl(
\frac{1}\varrho \biggr)\ge8 h(\varrho).
\]
From here, we obtain the
four conditions
\[
\varrho'> C^2 \varrho h(\varrho)^2,\qquad
\varrho'> C \frac{\kappa
}{\lambda_1} h(\varrho),\qquad \varrho'>
C' \frac{M}{\lambda
_1}\varrho h(\varrho)
\]
and
\[
\varrho' \ge\varrho\biggl(\frac{2} {C'}\frac{M\kappa}{(\lambda_1)^2}
\log \biggl(\frac{1}\varrho\biggr)\biggr)^{2/3},
\]
where $C=\frac{32}{\gamma_2 C_1 }$ and $C'=\frac{16}{(\gamma_2 C_1)^2}$.

For the third term, we have
\begin{eqnarray*} &&P\Biggl[\exists I, |I|= \varrho N,\exists J, |J|=
\varrho' N, \sum_{\mu=2}^M
S^\mu_3(J,J) <-\gamma_3 S^1(J,I)
\Biggr]
\\
&&\quad\leq \sum_{J \dvt |J|= \varrho' N} P\Biggl[\sum
_{\mu=2}^M S^\mu _3(J,J) <-
\gamma_3 C_1 \lambda_1 |J|\Biggr]
\\
&&\quad\leq\sum_{J \dvt |J|= \varrho' N} \exp\biggl(-\frac{1}2
\frac
{\gamma
_3 C_1 \lambda_1 |J|}{\lambda_{\{J,J\}}+ {M e\{J,J\}}/{(\gamma_3
C_1 \lambda_1 |J|)}}\biggr),
\end{eqnarray*}
where $\lambda_{\{J,J\}}=\lambda_1(E\{J,J\})$ is the largest eigenvalue
of the graph determined by the undirected edges in $E\{J,J\}$.
Using Lemma \ref{subgraph}, we have
\[
\lambda_{\{J,J\}} \le2 \varrho' \lambda_1 + 2
\kappa
\]
and $e\{J,J\}\leq(\varrho' \lambda_1 + \kappa) \varrho' N$, and we
obtain as for the previous terms
\begin{eqnarray*}
&&\exp \biggl(-\frac{1}2\frac{\gamma_3 C_1
\lambda_1 |J|}{\lambda_{\{J,J\}}+ {M e\{J,J\}}/{(\gamma_3 C_1
\lambda_1 |J|)}}
\biggr) \\
&&\quad\le \exp \biggl(-\frac{\gamma_3
C_1}2 \frac{ \varrho' N}{(2+{M}/{(\gamma_3 C_1 \lambda
_1)})\varrho'
+({\kappa}/{\lambda_1})(2+{M}/{(\gamma_3 C_1 \lambda_1)})
} \biggr).
\end{eqnarray*}
There are $ {N\choose|J|} $ ways to choose the set $J$.
From this, we see that
\[
P\Biggl[\exists I, |I|= \varrho N,\exists J, |J|= \varrho' N, \sum
_{\mu=2}^M S^\mu_3(J,J)
<-\gamma_3 S^1(J,I)\Biggr]
\]
becomes small, if
the condition
\[
\frac{\gamma_3 C_1}{2(2+{M}/{(\gamma_3 C_1 \lambda_1)})} \frac{1}{1
+ {\kappa}/{(\lambda_1\varrho')}} > h\bigl(\varrho'\bigr),
\]
is fulfilled,
which is true if
%
%e5.6 #&#
\begin{equation}
\label{nextconds} h\bigl(\varrho'\bigr)< C \quad\mbox{and}\quad h\bigl(
\varrho'\bigr)< C \varrho' \frac{\lambda_1}{\kappa}\qquad \mbox{where }
C=\frac{\gamma_3 C_1}{4(2+{M}/{(\gamma_3 C_1 \lambda_1)})}.
\end{equation}

As we assume that $M\leq c \lambda_1$,
there exists a $\varrho_2(\gamma_3, C_1)\in(0,0.1) $, such that the
first inequality in \eqref{nextconds} is true if $\varrho'<\varrho_2$.

Using the bound
$h(\varrho') \leq-2 \varrho' \log(\varrho')$ again, we get that there
exists $c>0$ such that the second condition in \eqref{nextconds} is
true if
\[
\varrho'> \exp\biggl(-c \frac{\lambda_1}{\kappa}\biggr).
\]

For the fourth term, we have
\begin{eqnarray*}
&&P\Biggl[\exists I, |I|= \varrho N,\exists J, |J|=
\varrho' N, \sum_{\mu=2}^M
S^\mu_4(J,I) >\gamma_4 S^1(J,I)
\Biggr]
\\
&&\quad\leq \sum_{I \dvt |I|= \varrho N}\sum_{J \dvt |J|= \varrho' N}
P\Biggl[\sum_{\mu=2}^M S^\mu_4(J,I)
>\gamma_4 C_1 \lambda_1 |J|\Biggr]
\\
&&\quad\leq\sum_{I \dvt |I|= \varrho N}\sum_{J \dvt |J|= \varrho' N}
\exp\biggl(-\frac{1}2\frac{\gamma_4 C_1 \lambda_1 |J|}{\lambda_{J\cap I}+
{M e\{J\cap I,J\cap I\}}/{(\gamma_4 C_1 \lambda_1 |J|})}\biggr),
\end{eqnarray*}
where $\lambda_{J\cap I}=\lambda_1(E\{J\cap I,J\cap I\})$ is the
largest eigenvalue of the graph determined by the undirected edges in
$E\{J\cap I,J\cap I\}$.

Using Lemma \ref{subgraph} and assuming that $\varrho'\le\varrho$, we
have $\lambda_{J\cap I} \le2 \varrho' \lambda_1 + 2 \kappa$ and
$e\{
J\cap I,J\cap I\}\leq(\varrho' \lambda_1 + \kappa) \varrho' N$, which
are the same bounds as for the third term. There are ${N\choose|I|}
{N\choose|J|} $ ways to choose the sets $I$ and $J$ and using again
\[
\pmatrix{N
\cr
|I|} \pmatrix{N
\cr
|J|} \leq\exp(\bigl(h(\varrho )+h\bigl(
\varrho'\bigr)N\bigr)\leq \exp\bigl(2 h(\varrho) N\bigr),
\]
we finally arrive at the same conditions as for the third term, with
possibly a different constant~$C$.

Finally, the various conditions can be summarized as
\begin{eqnarray*}
%$
\varrho'&\ge& c_2 \varrho \biggl(
\frac{\kappa}{\lambda_1}\biggr)^2,\qquad %$
%$
 \varrho_1\ge\varrho\ge\varrho' \ge\varrho_0,\qquad
% \exp(- c_2
%\frac{\lambda_1}{ \mu+ \frac{Mm}{\lambda_1}}),
%$
%$
 \varrho'>
C^2 \varrho h(\varrho)^2,
\\
\varrho'&>& C \frac{\kappa}{\lambda_1} h(\varrho),\qquad \varrho
'> C' \frac{M}{\lambda_1}\varrho h(\varrho) \quad\mbox{and}\quad %$
\varrho' \ge\varrho\biggl(\frac{2} {C'}
\frac{M\kappa}{(\lambda_1)^2} \log \biggl(\frac{1}\varrho\biggr)
\biggr)^{2/3}.
\end{eqnarray*}
Finally, taking into account all the conditions, we get that (\ref
{prob1}) is true if we choose
\[
f(\varrho)= \max\biggl\{c_1 \varrho \biggl(\frac{\kappa}{\lambda_1}
\biggr)^2, c_1 \varrho h(\varrho),c_1
\frac{\kappa}{\lambda_1} h(\varrho),c_1 \varrho \biggl(\frac{M\kappa}{(\lambda_1)^2}
\log\biggl(\frac{1}\varrho\biggr)\biggr)^{2/3}, \varrho
_0\biggr\}
\]
for some $c_1>0$ large enough and we see that $f(\varrho)\le\varrho$
if $\varrho\in(\varrho_0,\varrho_1)$ with $\varrho_1$ small enough.
\end{pf}

In order to prove the Theorem \ref{synchronous}, we will apply Lemma
\ref{mainlemma} repeatedly until the system attains an original
pattern. Using
\[
\varrho_0=\exp\biggl(-c_2\frac{\lambda_1}{\kappa+{Mm}/{\lambda_1}}\biggr),
\]
we get that the system can attain an original pattern, that is, $\varrho_0
N<1$, only if
\[
\kappa+ \frac{Mm}{\lambda_1} < c_2 \lambda_1/ \log(N)
\]
(which follows from the choice of $M$ made
in Theorem \ref{synchronous}).

To determine the maximal number of steps the synchronous dynamics needs
to converge, we analyze the following sequences.

%le5.2 #&#
\begin{lemma}\label{sequencelemma}
Let $(w_n)_{n\in\N}, (x_n)_{n\in\N}, (y_n)_{n\in\N}$ and
$(z_n)_{n\in\N}$ such that
\[
w_0=x_0= y_0 = z_0 = \varrho
\in\biggl[\exp\biggl(-\frac{1}{2c}\frac{\lambda_1}{
\kappa}\biggr), 1/e\biggr]
\]
and
\begin{eqnarray*}
w_{n+1}&=&c w_n \biggl(\frac{\kappa}{\lambda_1}
\biggr)^2,\qquad x_{n+1} = c x_n h(x_n),
\\
y_{n+1}&=& c\frac{\kappa}{\lambda_1} h(y_n) \quad\mbox{and}
\\
z_{n+1} &=& c z_n \biggl(\frac{M\kappa}{(\lambda_1)^2} \log\biggl(
\frac{1}{z_n}\biggr)\biggr)^{2/3},
\end{eqnarray*}
for $n\in\N$ and $c>0$. Let us assume that $\frac{\lambda_1}{\kappa
} >
C_1 \log N$ for some $C_1> 1$ large enough and that $M\le C_2\lambda_1$
for some $C_2>0$.
Then the sequences $(w_n), (x_n), (y_n)$ and $(z_n)$ are decreasing and
there exists $C_3>0$ and
\[
n_0\ge C_3 \max\biggl\{\log\log N, \frac{\log(N)}{\log({\lambda
_1}/{(\kappa\log N)})}
\biggr\}
\]
such that $\max\{w_{n_0}, x_{n_0}, y_{n_0}, z_{n_0}\}<1/N$.
\end{lemma}

\begin{pf}
Let us first consider the sequence $(w_n)$. Iterating $w_{n+1} = a
w_n$, with $a=c (\frac{\kappa}{\lambda_1})^2$, we get trivially $w_n =
a^n w_0$ from which we deduce that $w_n < \frac{1}N$ as soon as $n > c_1
\frac{\log(N)}{\log({\lambda_1}/{\kappa})}$ for some $c_1>0$.

For the sequence $(x_n)$, using $h(x)\leq-2 x \log(x)\leq2 \sqrt{x}$
for $x\in[0,1/2]$, we have
$x_{n+1}\leq(C x_n)^{3/2}$ for some constant $C>0$. Iterating, we get
$x_n \le(C^3 x_0)^{(3/2)^n}$, from which we deduce that
$x_n < \frac{1}N$ if $n\ge c_2 \log\log N$ for some $c_2>0$, if $x_0$
is small enough.

For the sequence $(y_n)$, using again $h(x)\leq-2 x \log(x)$, we have
to iterate the relation
$y_{n+1} = a y_n \log(\frac{1} {y_n})$, with $a=2c\frac{\kappa
}{\lambda_1}$.
If we consider $y_0 \in[\exp(-1/a), \exp(-1)] $,
the inductively defined sequence $y_{n+1} = g(y_n)$ is decreasing and
converges to $\exp(-1/a)$ since the function $g(x)= - a x \log(x)$ is
increasing on the interval $[\exp(-1/a), \exp(-1)]$, $y_1 \le y_0$ and
$\exp(-1/a)$ is the single fixed point of $g$.
Moreover, we have
\begin{eqnarray*} y_{n+2}&=& a^2 y_n
\log\biggl(\frac{1} {y_n}\biggr) \biggl( \log\biggl(\frac{1} {y_n}\biggr)
+ \log \biggl(\frac{1}a\biggr) + \log\biggl(\frac{1}{\log({1}/ {y_n})}\biggr)
\biggr)
\\
&\leq& a^2 y_n \log\biggl(\frac{1} {y_n}\biggr)
\biggl( \log\biggl(\frac{1} {y_n}\biggr) + \log\biggl(\frac{1}a
\biggr) \biggr),
\end{eqnarray*}
if $y_n\le1/e$.
By iteration, if we set $b=\log(\frac{1}{\min\{\varrho, a\}})$, we get
similarly for all $n\in\N$,
\begin{eqnarray*}y_{n}& \leq &a^n
y_0 \prod_{i=0}^{n-1} \biggl[
\log\biggl(\frac{1}{y_0}\biggr) + i \log\biggl(\frac{1}a \biggr)
\biggr]
\\
& \leq&(ab)^n y_0 n!
\\
& \leq& c_3 \biggl(ab \frac{n}e\biggr)^n
\sqrt{n}
\\
& =& c_3 \exp\biggl( n\bigl(\log(a) +\log(b)+ \log(n)-1\bigr) +
\frac{1}2\log( n)\biggr)
\\
& \le& c_3 \exp\bigl(-c_4 \log N \bigl(1+ \mathrm{o}(1)
\bigr)\bigr),
\end{eqnarray*}
for some $c_3>0$, if $n=c_4 \log(N)/(-\log a - \log\log N)$. In
particular, this justifies the hypothesis $\frac{\lambda_1}{\kappa} >
c_5 \log N$ for some $c_5> 1$ large enough.
We therefore see that there exists some $c_6>0$ such that
$\mathrm{e}^{-1/a} \le y_n <\frac{1}N$ for $ n = c_6 \frac{\log(N)}{\log(
{\lambda_1}/{(\kappa\log N)})}$.\

The third sequence can be rewritten as $z_{n+1} = a z_n (\log\frac{1}{z_n})^{2/3}$, with $a=c (\frac{M}{\lambda_1} \frac{\kappa
}{\lambda
_1})^{2/3}$. With the same technique as for $y_n$, we get that the
sequence $(z_n)$ converges to $\exp(-1/a^{3/2})$ and $z_n <1/N$ if
\[
n \ge c_7 {\log(N)}\Big/\log\biggl({\frac{(\lambda_1)^2} {M\kappa\log(N)}}\biggr).
\]

This proves the lemma.
\end{pf}

The combination of the previous considerations and Lemma \ref
{sequencelemma} then yields the Theorem \ref{synchronous}.

\begin{appendix}\label{app}
%s6 #&#
\section*{Appendix: On the degrees of the Erd\"os--Renyi graph}\label{appA}
To prove the Corollary \ref{GNp}, we need to estimate the minimum and
the maximum degrees of a typical Erd\"os--Renyi graph $G(N,p)$.
The following result could not be found in the literature. We prove in
this appendix the following.

%le6.1 #&#
%\setcounter{lemma}{0}
\begin{lemmas} If $G$ be is chosen randomly according to the model
$G(N,p)$, then if $p \gg\frac{\log N}N$,
%and $1-p \ge c(\frac{\log N}N)^{1/3}$ for some constant $c>0$ large
%enough,
for a set of realizations of $G$ the probability of which converges to
one as $N \to\infty$, we have
$ m= (1+\mathrm{o}(1)) Np$ and $ \delta= (1+\mathrm{o}(1)) Np$.
\end{lemmas}

\begin{pf}
Let $G$ chosen randomly according to the model $G(N,p)$. The law of the
degree $d_i$ of an arbitrary vertex $i$ of $G$ is the binomial
distribution $B(N,p)$. Hence, using the exponential Markov inequality,
we arrive at the following bound: for $p<a<1$ and $N\ge1$,
\[
P[d_i \ge a N] \leq\exp\bigl(-N H(a,p)\bigr),
\]
where $H$ is the relative entropy or Kullback--Leibler information
\[
H(a,p) = a \log\biggl(\frac{a}p\biggr) + (1-a) \log\biggl(
\frac{1-a}{1-p}\biggr).
\]

If we now set $m=\max_i d_i$ as above, we obtain
\[
P[m \ge a N] \leq\sum_{i=1}^N
P[d_i \ge a N] \le N \exp\bigl(-N H(a,p)\bigr).
\]
If we choose $a=(1+\varepsilon)p$, for some $\varepsilon>0$ such that
$a<1$, we therefore get
\[
P\bigl[m \ge(1+\varepsilon)p N\bigr] \leq N \exp \biggl(-N \biggl(p(1+
\varepsilon) \log (1+\varepsilon)+\bigl(1-(1+\varepsilon)p\bigr)\log\biggl(
\frac{1-(1+\varepsilon
)p}{1-p}\biggr)\biggr) \biggr).
\]
Moreover, we have $(1-(1+\varepsilon)p)\log(\frac{1-(1+\varepsilon
)p}{1-p}) \ge- p \varepsilon$.

Indeed, if we set $q=1-p$ and $u= 1- (1+\varepsilon)p = q -
\varepsilon
p$, the last inequality is equivalent to
$\log( q/u) \le q/u -1$ which is true since $ q/u >1$.
Thus,
\begin{eqnarray*}
P\bigl[m \ge(1+\varepsilon)p N\bigr] &\leq& N \exp\bigl(-Np \bigl((1+\varepsilon)
\log (1+\varepsilon)-\varepsilon\bigr)\bigr)\\
& \leq& N \exp\biggl(-Np
\frac{\varepsilon
^2}2 \bigl(1+\mathrm{o}(1)\bigr)\biggr),
\end{eqnarray*}
if we suppose that $\varepsilon=\mathrm{o}(1)$ as $N\rightarrow\infty$.
Choosing $ \varepsilon= 2 \sqrt{\frac{\log N}{pN}}$, we
have $\varepsilon=\mathrm{o}(1)$ for $p \gg\frac{\log N}N$, and
\[
P\bigl[m \ge(1+\varepsilon)p N\bigr] \rightarrow0 \qquad\mbox{as }N\rightarrow \infty.
\]

Moreover, we have $m\ge\lambda_1$ and $\lambda_1 = (1+\mathrm{o}(1))p N$
(with
probability converging to 1 as $N\rightarrow\infty$), which gives the
result for $m$.

Now, if we set $\delta:=\min_i d_i$, we want to prove that $P[\delta
\ge(1+\varepsilon') pN] \rightarrow1$, as $N\rightarrow\infty$, for
some $\varepsilon'=\mathrm{o}(1)$.
%If we assume that $p \gg\log(N)/N$ and $1-p \gg\log(N)/N$,
We consider
the complementary graph, that is, the random graph $\overline G$, such that
exactly those edges are missing in a realization of $\overline G$, that
occur in the corresponding realization of the original random graph
$G$. Now the maximum degree $\overline m$ of $\overline G$ and the
minimum degree $\delta$ of $G$ are linked via the relation
$\delta= N-1- \overline m$.

As $\overline G$ is chosen randomly according to the model $G(N,1-p)$,
we have
\[
P\bigl[\overline m \ge(1+\varepsilon) (1-p) N\bigr] \le N \exp\bigl(-N H
\bigl((1+\epsilon ) (1-p),1-p\bigr)\bigr),
\]
for all $\epsilon>0$ such that $(1+\varepsilon)(1-p)<1$.
Now
\[
H\bigl((1+\varepsilon) (1-p),1-p\bigr) = (1+\varepsilon) (1-p) \log (1+
\varepsilon) +(p-\varepsilon+p\varepsilon)\log\biggl(1-\frac{\varepsilon(1-p)}p\biggr).
\]
If we suppose that $\varepsilon= \mathrm{o}(1)$ and $\varepsilon\ll p$, using
the inequality
$\log(1-x) \ge-x - x^2/2-x^3$ for $x\in(0,1/2)$ to bound the last
term, we obtain the estimate
\begin{eqnarray*}
H\bigl((1+\varepsilon) (1-p),1-p\bigr) \ge\frac{\varepsilon^2}{2p} \biggl((1-p) - C
\biggl( \varepsilon+ \frac{\varepsilon}p\biggr) + \mathcal{O}\bigl(p
\varepsilon^2\bigr) \biggr),
\end{eqnarray*}
for some $C>0$ and
\[
P\bigl[\overline m \ge(1+\varepsilon) (1-p) N\bigr] \le\exp\biggl( -N
\frac
{\varepsilon
^2}{2p} \biggl((1-p) - C \biggl(\varepsilon+ \frac{\varepsilon}p\biggr)
+ \mathcal {O}\bigl(p\varepsilon^2\bigr) \biggr) + \log(N)\biggr).
\]
There exists some $c>0$ such that if we choose
$\varepsilon= \sqrt{ \frac{4p}{c(1-p)}\frac{\log(N)}N }$, we get
\[
P\bigl[\overline m \ge(1+\varepsilon) (1-p) N\bigr] \le\exp\biggl( -c N
\frac
{\varepsilon^2}{2p} (1-p)+ \log(N)\biggr) \rightarrow0,
\]
under the conditions
$p \gg\frac{\log N}N$ and $1-p \gg(\frac{\log N}N)^{1/3}$.

Finally, we get $\delta\ge N-1 - (1+ \varepsilon) (1-p)N = (1+\mathrm{o}(1))
Np$, which is the result under these two conditions.

Eventually, we will extend this result for all $p$ such that
$p\rightarrow1$, as $N\rightarrow+\infty$.
As previously, using the exponential Markov inequality, we get the
following bound: for $0<b<p<1$, and $N\ge1$,
\[
P[d_i \le b N] \leq\exp\bigl(-N H(b,p)\bigr).
\]
We set $p=1-a_N$ and $ b= 1- b_N$, for some strictly positive sequences
$(a_N)$ and $(b_N)$ such that
$a_N + b_N \rightarrow0$, as $N\rightarrow\infty$, $a_N \ll b_N$, and
we can restrict to the case
$a_N < (c\frac{\log N}N)^{1/3}$ for some $c>0$.
We get
\begin{eqnarray*}
P[\delta\le b N] &\leq &N \exp\biggl(-N \biggl( (1-b_N) \log\biggl(
\frac
{1-b_N}{1-a_N}\biggr) + b_N \log\biggl(\frac{b_N}{a_N}\biggr)
\biggr)\biggr)
\\
&\leq&\exp\biggl(-N \biggl( b_N \log\biggl(\frac{b_N}{a_N}
\biggr)-2b_N\biggr)+\log(N)\biggr).
\end{eqnarray*}
So, we need to choose $b_N$ such that
\[
b_N \log\biggl(\frac{b_N}{a_N}\biggr) > \frac{\log(N)}{N}.
\]
We have
\[
b_N \log\biggl(\frac{b_N}{a_N}\biggr) > b_N \log
\biggl({b_N}\biggl(\frac{N}{c \log(N)}\biggr)^{1/3}\biggr)>
\frac{\log(N)}{N},
\]
if we choose for instance $b_N =(\frac{\log N}N)^{\gamma}$ with
$\gamma
\in(0,1/3)$.

Finally, we get for all $p\rightarrow1$ that $\delta\ge(1-b_N) N =
(1+ \mathrm{o}(1)) Np$, with probability converging to 1 as $N\rightarrow\infty$.
\end{pf}
\end{appendix}
%
%\bibliographystyle{abbrv}
%\bibliography{LiteraturDatenbank}

\begin{thebibliography}{40}
% pybtex-1.11. Style name=bej, version=1.4, label_style=nolabel, sorting_style=complex, cfg=None, language=None.


%b1 ###
%b1 #&#
\bibitem{AGS}
\begin{barticle}[mr]
\bauthor{\bsnm{Amit},~\bfnm{Daniel~J.}\binits{D.J.}},
\bauthor{\bsnm{Gutfreund},~\bfnm{Hanoch}\binits{H.}} \AND
\bauthor{\bsnm{Sompolinsky},~\bfnm{H.}\binits{H.}}
(\byear{1985}).
\btitle{Spin-glass models of neural networks}.
\bjournal{Phys. Rev. A (3)}
\bvolume{32}
\bpages{1007--1018}.
\bid{doi={10.1103/PhysRevA.32.1007}, issn={1050-2947}, mr={0797031}}
\end{barticle}
\bptok{imsref}%
% NOT OUTPUTED:
%   issn = 1050-2947
%   url = http://dx.doi.org/10.1103/PhysRevA.32.1007
%   number = 2
%   coden = PLRAAN
%   fjournal = Physical Review. A. Third Series
\endbibitem

%b2 ###
%b2 #&#
\bibitem{bassett}
\begin{barticle}[auto:STB|2014/06/10|07:15:57]
\bauthor{\bsnm{Bassett},~\bfnm{D.}\binits{D.}} \AND
\bauthor{\bsnm{Bullmore},~\bfnm{E.}\binits{E.}}
(\byear{2006}).
\btitle{Small-world brain networks}.
\bjournal{Neuroscientist}
\bvolume{12}
\bpages{512--523}.
\end{barticle}
\bptok{imsref}%
% NOT OUTPUTED:
%   number = 6
\endbibitem

%b3 ###
%b3 #&#
\bibitem{Bov97}
\begin{barticle}[mr]
\bauthor{\bsnm{Bovier},~\bfnm{Anton}\binits{A.}}
(\byear{1999}).
\btitle{Sharp upper bounds on perfect retrieval in the {H}opfield model}.
\bjournal{J. Appl. Probab.}
\bvolume{36}
\bpages{941--950}.
\bid{issn={0021-9002}, mr={1737066}}
\end{barticle}
\bptok{imsref}%
% NOT OUTPUTED:
%   issn = 0021-9002
%   number = 3
%   coden = JPRBAM
%   fjournal = Journal of Applied Probability
\endbibitem

%b4 ###
%b4 #&#
\bibitem{bovierbook}
\begin{bbook}[mr]
\bauthor{\bsnm{Bovier},~\bfnm{Anton}\binits{A.}}
(\byear{2006}).
\btitle{Statistical Mechanics of Disordered Systems: A Mathematical Perspective}.
\bseries{Cambridge Series in Statistical and Probabilistic Mathematics}.
\blocation{Cambridge}:
\bpublisher{Cambridge Univ. Press}.
\bid{doi={10.1017/CBO9780511616808}, mr={2252929}}
\end{bbook}
\bptok{imsref}%
% NOT OUTPUTED:
%   isbn = 978-0-521-84991-3; 0-521-84991-8
%   url = http://dx.doi.org/10.1017/CBO9780511616808
%   fpage = xiv+312
\endbibitem

%b5 ###
%b5 #&#
\bibitem{BG92}
\begin{barticle}[mr]
\bauthor{\bsnm{Bovier},~\bfnm{Anton}\binits{A.}} \AND
\bauthor{\bsnm{Gayrard},~\bfnm{V{\'e}ronique}\binits{V.}}
(\byear{1992}).
\btitle{Rigorous bounds on the storage capacity of the dilute {H}opfield model}.
\bjournal{J. Stat. Phys.}
\bvolume{69}
\bpages{597--627}.
\bid{doi={10.1007/BF01050427}, issn={0022-4715}, mr={1193852}}
\end{barticle}
\bptok{imsref}%
% NOT OUTPUTED:
%   issn = 0022-4715
%   url = http://dx.doi.org/10.1007/BF01050427
%   number = 3-4
%   coden = JSTPSB
%   fjournal = Journal of Statistical Physics
\endbibitem

%b6 ###
%b6 #&#
\bibitem{BG93}
\begin{barticle}[mr]
\bauthor{\bsnm{Bovier},~\bfnm{Anton}\binits{A.}} \AND
\bauthor{\bsnm{Gayrard},~\bfnm{V{\'e}ronique}\binits{V.}}
(\byear{1993}).
\btitle{Rigorous results on the thermodynamics of the dilute {H}opfield model}.
\bjournal{J. Stat. Phys.}
\bvolume{72}
\bpages{79--112}.
\bid{doi={10.1007/BF01048041}, issn={0022-4715}, mr={1233027}}
\end{barticle}
\bptok{imsref}%
% NOT OUTPUTED:
%   issn = 0022-4715
%   url = http://dx.doi.org/10.1007/BF01048041
%   number = 1-2
%   coden = JSTPSB
%   fjournal = Journal of Statistical Physics
\endbibitem

%b7 ###
%b7 #&#
\bibitem{Burshtein}
\begin{barticle}[mr]
\bauthor{\bsnm{Burshtein},~\bfnm{David}\binits{D.}}
(\byear{1994}).
\btitle{Nondirect convergence radius and number of iterations of the {H}opfield associative memory}.
\bjournal{IEEE Trans. Inform. Theory}
\bvolume{40}
\bpages{838--847}.
\bid{doi={10.1109/18.335894}, issn={0018-9448}, mr={1295316}}
\end{barticle}
\bptok{imsref}%
% NOT OUTPUTED:
%   issn = 0018-9448
%   url = http://dx.doi.org/10.1109/18.335894
%   number = 3
%   coden = IETTAW
%   fjournal = Institute of Electrical and Electronics Engineers. Transactions on Information Theory
\endbibitem

%b8 ###
%b8 #&#
\bibitem{ChenAdams}
\begin{bmisc}[auto:STB|2014/06/10|07:15:57]
\bauthor{\bsnm{Chen},~\bfnm{W.}\binits{W.}},
\bauthor{\bsnm{Adams},~\bfnm{R.}\binits{R.}},
\bauthor{\bsnm{Calcraft},~\bfnm{L.}\binits{L.}},
\bauthor{\bsnm{Steuber},~\bfnm{V.}\binits{V.}} \AND
\bauthor{\bsnm{Davey},~\bfnm{N.}\binits{N.}}
(\byear{2008}).
\bhowpublished{Connectivity graphs and the performance of sparse associative
memory models. In \textit{IEEE World Congress on Computational Intelligence
(Conference IJCNN 2008)} 2742--2749}.
\end{bmisc}
\bptok{imsref}%
% NOT OUTPUTED:
%   sortkey = Chen(2008
\endbibitem

%b9 ###
%b9 #&#
\bibitem{Chung-Lu02}
\begin{barticle}[mr]
\bauthor{\bsnm{Chung},~\bfnm{Fan}\binits{F.}} \AND
\bauthor{\bsnm{Lu},~\bfnm{Linyuan}\binits{L.}}
(\byear{2002}).
\btitle{The average distances in random graphs with given expected degrees}.
\bjournal{Proc. Natl. Acad. Sci. USA}
\bvolume{99}
\bpages{15879--15882 (electronic)}.
\bid{doi={10.1073/pnas.252631999}, issn={1091-6490}, mr={1944974}}
\end{barticle}
\bptok{imsref}%
% NOT OUTPUTED:
%   issn = 1091-6490
%   url = http://dx.doi.org/10.1073/pnas.252631999
%   number = 25
%   coden = PNASFB
%   fjournal = Proceedings of the National Academy of Sciences of the United States of America
\endbibitem

%b10 ###
%b10 #&#
\bibitem{ChungLu2002AC}
\begin{barticle}[mr]
\bauthor{\bsnm{Chung},~\bfnm{Fan}\binits{F.}} \AND
\bauthor{\bsnm{Lu},~\bfnm{Linyuan}\binits{L.}}
(\byear{2002}).
\btitle{Connected components in random graphs with given expected degree sequences}.
\bjournal{Ann. Comb.}
\bvolume{6}
\bpages{125--145}.
\bid{doi={10.1007/PL00012580}, issn={0218-0006}, mr={1955514}}
\end{barticle}
\bptok{imsref}%
% NOT OUTPUTED:
%   issn = 0218-0006
%   url = http://dx.doi.org/10.1007/PL00012580
%   number = 2
%   fjournal = Annals of Combinatorics
\endbibitem

%b11 ###
%b11 #&#
\bibitem{Chung-Radcliffe11}
\begin{barticle}[mr]
\bauthor{\bsnm{Chung},~\bfnm{Fan}\binits{F.}} \AND
\bauthor{\bsnm{Radcliffe},~\bfnm{Mary}\binits{M.}}
(\byear{2011}).
\btitle{On the spectra of general random graphs}.
\bjournal{Electron. J. Combin.}
\bvolume{18}
\bpages{Paper 215, 14}.
\bid{issn={1077-8926}, mr={2853072}}
\end{barticle}
\bptok{imsref}%
% NOT OUTPUTED:
%   issn = 1077-8926
%   number = 1
%   fjournal = Electronic Journal of Combinatorics
\endbibitem

%b12 ###
%b12 #&#
\bibitem{diaconisstroock}
\begin{barticle}[mr]
\bauthor{\bsnm{Diaconis},~\bfnm{Persi}\binits{P.}} \AND
\bauthor{\bsnm{Stroock},~\bfnm{Daniel}\binits{D.}}
(\byear{1991}).
\btitle{Geometric bounds for eigenvalues of {M}arkov chains}.
\bjournal{Ann. Appl. Probab.}
\bvolume{1}
\bpages{36--61}.
\bid{issn={1050-5164}, mr={1097463}}
\end{barticle}
\bptok{imsref}%
% NOT OUTPUTED:
%   issn = 1050-5164
%   url = http://links.jstor.org/sici?sici=1050-5164(199102)1:1<36:GBFEOM>2.0.CO;2-E&origin=MSN
%   number = 1
%   fjournal = The Annals of Applied Probability
\endbibitem

%b13 ###
%b13 #&#
\bibitem{durrett_graphs}
\begin{bbook}[mr]
\bauthor{\bsnm{Durrett},~\bfnm{Rick}\binits{R.}}
(\byear{2010}).
\btitle{Random Graph Dynamics}.
\bseries{Cambridge Series in Statistical and Probabilistic Mathematics}.
\blocation{Cambridge}:
\bpublisher{Cambridge Univ. Press}.
\bid{mr={2656427}}
\end{bbook}
\bptok{imsref}%
% NOT OUTPUTED:
%   isbn = 978-0-521-15016-3
%   fpage = x+210
\endbibitem

%b14 ###
%b14 #&#
\bibitem{Feige2005}
\begin{barticle}[mr]
\bauthor{\bsnm{Feige},~\bfnm{Uriel}\binits{U.}} \AND
\bauthor{\bsnm{Ofek},~\bfnm{Eran}\binits{E.}}
(\byear{2005}).
\btitle{Spectral techniques applied to sparse random graphs}.
\bjournal{Random Structures Algorithms}
\bvolume{27}
\bpages{251--275}.
\bid{doi={10.1002/rsa.20089}, issn={1042-9832}, mr={2155709}}
\end{barticle}
\bptok{imsref}%
% NOT OUTPUTED:
%   issn = 1042-9832
%   url = http://dx.doi.org/10.1002/rsa.20089
%   number = 2
%   fjournal = Random Structures \& Algorithms
\endbibitem

%b15 ###
%b15 #&#
\bibitem{Piccoetal}
\begin{barticle}[mr]
\bauthor{\bsnm{Ferrari},~\bfnm{Pablo~A.}\binits{P.A.}},
\bauthor{\bsnm{Mart{\'{\i}}nez},~\bfnm{Servet}\binits{S.}} \AND
\bauthor{\bsnm{Picco},~\bfnm{Pierre}\binits{P.}}
(\byear{1992}).
\btitle{A lower bound for the memory capacity in the {P}otts--{H}opfield model}.
\bjournal{J. Stat. Phys.}
\bvolume{66}
\bpages{1643--1652}.
\bid{doi={10.1007/BF01054440}, issn={0022-4715}, mr={1156420}}
\end{barticle}
\bptok{imsref}%
% NOT OUTPUTED:
%   issn = 0022-4715
%   url = http://dx.doi.org/10.1007/BF01054440
%   number = 5-6
%   coden = JSTPSB
%   fjournal = Journal of Statistical Physics
\endbibitem

%b16 ###
%b16 #&#
\bibitem{furedikomlos}
\begin{barticle}[mr]
\bauthor{\bsnm{F{\"u}redi},~\bfnm{Z.}\binits{Z.}} \AND
\bauthor{\bsnm{Koml{\'o}s},~\bfnm{J.}\binits{J.}}
(\byear{1981}).
\btitle{The eigenvalues of random symmetric matrices}.
\bjournal{Combinatorica}
\bvolume{1}
\bpages{233--241}.
\bid{doi={10.1007/BF02579329}, issn={0209-9683}, mr={0637828}}
\end{barticle}
\bptok{imsref}%
% NOT OUTPUTED:
%   issn = 0209-9683
%   url = http://dx.doi.org/10.1007/BF02579329
%   number = 3
%   coden = COMBDI
%   fjournal = Combinatorica. An International Journal of the J\'anos Bolyai Mathematical Society
\endbibitem

%b17 ###
%b17 #&#
\bibitem{gentzloewe}
\begin{barticle}[mr]
\bauthor{\bsnm{Gentz},~\bfnm{Barbara}\binits{B.}} \AND
\bauthor{\bsnm{L{\"o}we},~\bfnm{Matthias}\binits{M.}}
(\byear{1999}).
\btitle{The fluctuations of the overlap in the {H}opfield model with finitely many patterns at the critical temperature}.
\bjournal{Probab. Theory Related Fields}
\bvolume{115}
\bpages{357--381}.
\bid{doi={10.1007/s004400050241}, issn={0178-8051}, mr={1725405}}
\end{barticle}
\bptok{imsref}%
% NOT OUTPUTED:
%   issn = 0178-8051
%   url = http://dx.doi.org/10.1007/s004400050241
%   number = 3
%   coden = PTRFEU
%   fjournal = Probability Theory and Related Fields
\endbibitem

%b18 ###
%b18 #&#
\bibitem{Hopfield1982}
\begin{barticle}[mr]
\bauthor{\bsnm{Hopfield},~\bfnm{J.~J.}\binits{J.J.}}
(\byear{1982}).
\btitle{Neural networks and physical systems with emergent collective computational abilities}.
\bjournal{Proc. Natl. Acad. Sci. USA}
\bvolume{79}
\bpages{2554--2558}.
\bid{issn={0027-8424}, mr={0652033}}
\end{barticle}
\bptok{imsref}%
% NOT OUTPUTED:
%   issn = 0027-8424
%   number = 8
%   coden = PNASA6
%   fjournal = Proceedings of the National Academy of Sciences of the United States of America
\endbibitem

%b19 ###
%b19 #&#
\bibitem{JacobMorters}
\begin{bmisc}[auto:STB|2014/06/10|07:15:57]
\bauthor{\bsnm{Jacob},~\bfnm{E.}\binits{E.}} \AND
\bauthor{\bsnm{M{\"o}rters},~\bfnm{P.}\binits{P.}}
(\byear{2012}).
\bhowpublished{Spatial preferential attachment networks: Power laws
and clustering coefficients. Preprint.
Available at \arxivurl{arXiv:1210.3830v1}.}
\end{bmisc}
\bptok{imsref}%
% NOT OUTPUTED:
%   sortkey = Jacob(2012
\endbibitem

%b20 ###
%b20 #&#
\bibitem{jeongetal}
\begin{barticle}[auto:STB|2014/06/10|07:15:57]
\bauthor{\bsnm{Jeong},~\bfnm{H.}\binits{H.}},
\bauthor{\bsnm{Tombor},~\bfnm{B.}\binits{B.}},
\bauthor{\bsnm{Albert},~\bfnm{R.}\binits{R.}},
\bauthor{\bsnm{Oltvai},~\bfnm{Z.~N.}\binits{Z.N.}} \AND
\bauthor{\bsnm{Barabasi},~\bfnm{A.~L.}\binits{A.L.}}
(\byear{2000}).
\btitle{The large-scale organization of metabolic networks}.
\bjournal{Nature}
\bvolume{407}
\bpages{651--654}.
\end{barticle}
\bptok{imsref}%
% NOT OUTPUTED:
%   number = 6804
\endbibitem

%b21 ###
%b21 #&#
\bibitem{KomlosPaturi1988}
\begin{barticle}[auto:STB|2014/06/10|07:15:57]
\bauthor{\bsnm{Koml{\'o}s},~\bfnm{J.}\binits{J.}} \AND
\bauthor{\bsnm{Paturi},~\bfnm{R.}\binits{R.}}
(\byear{1988}).
\btitle{Convergence results in an associative memory model}.
\bjournal{Neural Networks}
\bvolume{1}
\bpages{239--250}.
\end{barticle}
\bptok{imsref}%
% NOT OUTPUTED:
%   number = 3
\endbibitem

%b22 ###
%b22 #&#
\bibitem{Komlos1993}
\begin{barticle}[mr]
\bauthor{\bsnm{Koml{\'o}s},~\bfnm{J{\'a}nos}\binits{J.}} \AND
\bauthor{\bsnm{Paturi},~\bfnm{Ramamohan}\binits{R.}}
(\byear{1993}).
\btitle{Effect of connectivity in an associative memory model}.
\bjournal{J. Comput. System Sci.}
\bvolume{47}
\bpages{350--373}.
%\bnote{29th Annual IEEE Symposium on Foundations of Computer Science (White Plains, NY, 1988)}.
\bid{doi={10.1016/0022-0000(93)90036-V}, issn={0022-0000}, mr={1246181}}
\end{barticle}
\bptok{imsref}%
% NOT OUTPUTED:
%   issn = 0022-0000
%   url = http://dx.doi.org/10.1016/0022-0000(93)90036-V
%   number = 2
%   coden = JCSSBM
%   fjournal = Journal of Computer and System Sciences
\endbibitem

%b23 ###
%b23 #&#
\bibitem{Krivelevich2003}
\begin{barticle}[mr]
\bauthor{\bsnm{Krivelevich},~\bfnm{Michael}\binits{M.}} \AND
\bauthor{\bsnm{Sudakov},~\bfnm{Benny}\binits{B.}}
(\byear{2003}).
\btitle{The largest eigenvalue of sparse random graphs}.
\bjournal{Combin. Probab. Comput.}
\bvolume{12}
\bpages{61--72}.
\bid{doi={10.1017/S0963548302005424}, issn={0963-5483}, mr={1967486}}
\end{barticle}
\bptok{imsref}%
% NOT OUTPUTED:
%   issn = 0963-5483
%   url = http://dx.doi.org/10.1017/S0963548302005424
%   number = 1
%   fjournal = Combinatorics, Probability and Computing
\endbibitem

%b24 ###
%b24 #&#
\bibitem{loukianova}
\begin{barticle}[mr]
\bauthor{\bsnm{Loukianova},~\bfnm{Daria}\binits{D.}}
(\byear{1994}).
\btitle{Capacit\'e de m\'emoire dans le mod\`ele de {H}opfield}.
\bjournal{C. R. Acad. Sci. Paris S\'er. I Math.}
\bvolume{318}
\bpages{157--160}.
\bid{issn={0764-4442}, mr={1260330}}
\end{barticle}
\bptok{imsref}%
% NOT OUTPUTED:
%   issn = 0764-4442
%   number = 2
%   coden = CASMEI
%   fjournal = Comptes Rendus de l'Acad\'emie des Sciences. S\'erie I. Math\'ematique
\endbibitem

%b25 ###
%b25 #&#
\bibitem{L98}
\begin{barticle}[mr]
\bauthor{\bsnm{L{\"o}we},~\bfnm{Matthias}\binits{M.}}
(\byear{1998}).
\btitle{On the storage capacity of {H}opfield models with correlated patterns}.
\bjournal{Ann. Appl. Probab.}
\bvolume{8}
\bpages{1216--1250}.
\bid{doi={10.1214/aoap/1028903378}, issn={1050-5164}, mr={1661188}}
\end{barticle}
\bptok{imsref}%
% NOT OUTPUTED:
%   issn = 1050-5164
%   url = http://dx.doi.org/10.1214/aoap/1028903378
%   number = 4
%   fjournal = The Annals of Applied Probability
\endbibitem

%b26 ###
%b26 #&#
\bibitem{LV_BEG}
\begin{barticle}[mr]
\bauthor{\bsnm{L{\"o}we},~\bfnm{Matthias}\binits{M.}} \AND
\bauthor{\bsnm{Vermet},~\bfnm{Franck}\binits{F.}}
(\byear{2005}).
\btitle{The storage capacity of the {B}lume--{E}mery--{G}riffiths neural network}.
\bjournal{J. Phys. A}
\bvolume{38}
\bpages{3483--3503}.
\bid{doi={10.1088/0305-4470/38/16/002}, issn={0305-4470}, mr={2131428}}
\end{barticle}
\bptok{imsref}%
% NOT OUTPUTED:
%   issn = 0305-4470
%   url = http://dx.doi.org/10.1088/0305-4470/38/16/002
%   number = 16
%   coden = JPHAC5
%   fjournal = Journal of Physics. A. Mathematical and General
\endbibitem

%b27 ###
%b27 #&#
\bibitem{LV_05}
\begin{barticle}[mr]
\bauthor{\bsnm{L{\"o}we},~\bfnm{Matthias}\binits{M.}} \AND
\bauthor{\bsnm{Vermet},~\bfnm{Franck}\binits{F.}}
(\byear{2005}).
\btitle{The storage capacity of the {H}opfield model and moderate deviations}.
\bjournal{Statist. Probab. Lett.}
\bvolume{75}
\bpages{237--248}.
\bid{doi={10.1016/j.spl.2005.06.001}, issn={0167-7152}, mr={2212355}}
\end{barticle}
\bptok{imsref}%
% NOT OUTPUTED:
%   issn = 0167-7152
%   url = http://dx.doi.org/10.1016/j.spl.2005.06.001
%   number = 4
%   coden = SPLTDC
%   fjournal = Statistics \& Probability Letters
\endbibitem

%%b28 ###
%\bibitem{LV07}
%\begin{barticle}[mr]
%\bauthor{\bsnm{L{\"o}we},~\bfnm{Matthias}\binits{M.}} \AND
%\bauthor{\bsnm{Vermet},~\bfnm{Franck}\binits{F.}}
%(\byear{2007}).
%\btitle{The capacity of {$q$}-state {P}otts neural networks with parallel retrieval dynamics}.
%\bjournal{Statist. Probab. Lett.}
%\bvolume{77}
%\bpages{1505--1514}.
%\bid{doi={10.1016/j.spl.2007.03.030}, issn={0167-7152}, mr={2395600}}
%\end{barticle}
%\bptok{imsref}%
%% NOT OUTPUTED:
%%   issn = 0167-7152
%%   url = http://dx.doi.org/10.1016/j.spl.2007.03.030
%%   number = 14
%%   coden = SPLTDC
%%   fjournal = Statistics \& Probability Letters
%\endbibitem

%b29 ###
%b28 #&#
\bibitem{LoweVermetqpotts}
\begin{barticle}[mr]
\bauthor{\bsnm{L{\"o}we},~\bfnm{Matthias}\binits{M.}} \AND
\bauthor{\bsnm{Vermet},~\bfnm{Franck}\binits{F.}}
(\byear{2007}).
\btitle{The capacity of {$q$}-state {P}otts neural networks with parallel retrieval dynamics}.
\bjournal{Statist. Probab. Lett.}
\bvolume{77}
\bpages{1505--1514}.
\bid{doi={10.1016/j.spl.2007.03.030}, issn={0167-7152}, mr={2395600}}
\end{barticle}
\bptok{imsref}%
% NOT OUTPUTED:
%   issn = 0167-7152
%   url = http://dx.doi.org/10.1016/j.spl.2007.03.030
%   number = 14
%   coden = SPLTDC
%   fjournal = Statistics \& Probability Letters
\endbibitem

%b30 ###
%b29 #&#
\bibitem{LV10}
\begin{barticle}[mr]
\bauthor{\bsnm{L{\"o}we},~\bfnm{Matthias}\binits{M.}} \AND
\bauthor{\bsnm{Vermet},~\bfnm{Franck}\binits{F.}}
(\byear{2011}).
\btitle{The {H}opfield model on a sparse {E}rd{\H o}s--{R}enyi graph}.
\bjournal{J. Stat. Phys.}
\bvolume{143}
\bpages{205--214}.
\bid{doi={10.1007/s10955-011-0167-1}, issn={0022-4715}, mr={2787981}}
\end{barticle}
\bptok{imsref}%
% NOT OUTPUTED:
%   issn = 0022-4715
%   url = http://dx.doi.org/10.1007/s10955-011-0167-1
%   number = 1
%   fjournal = Journal of Statistical Physics
\endbibitem

%b31 ###
%b30 #&#
\bibitem{LuPeng}
\begin{barticle}[mr]
\bauthor{\bsnm{Lu},~\bfnm{Linyuan}\binits{L.}} \AND
\bauthor{\bsnm{Peng},~\bfnm{Xing}\binits{X.}}
(\byear{2013}).
\btitle{Spectra of edge-independent random graphs}.
\bjournal{Electron. J. Combin.}
\bvolume{20}
\bpages{1--18}.
\bid{issn={1077-8926}, mr={3158266}}
\bptnote{check year}%
\end{barticle}
\bptok{imsref}%
% NOT OUTPUTED:
%   issn = 1077-8926
%   number = 4
%   fjournal = Electronic Journal of Combinatorics
\endbibitem

%b32 ###
%b31 #&#
\bibitem{MPRV}
\begin{barticle}[mr]
\bauthor{\bsnm{McEliece},~\bfnm{Robert~J.}\binits{R.J.}},
\bauthor{\bsnm{Posner},~\bfnm{Edward~C.}\binits{E.C.}},
\bauthor{\bsnm{Rodemich},~\bfnm{Eugene~R.}\binits{E.R.}} \AND
\bauthor{\bsnm{Venkatesh},~\bfnm{Santosh~S.}\binits{S.S.}}
(\byear{1987}).
\btitle{The capacity of the {H}opfield associative memory}.
\bjournal{IEEE Trans. Inform. Theory}
\bvolume{33}
\bpages{461--482}.
\bid{doi={10.1109/TIT.1987.1057328}, issn={0018-9448}, mr={0901677}}
\end{barticle}
\bptok{imsref}%
% NOT OUTPUTED:
%   issn = 0018-9448
%   url = http://dx.doi.org/10.1109/TIT.1987.1057328
%   number = 4
%   coden = IETTAW
%   fjournal = Institute of Electrical and Electronics Engineers. Transactions on Information Theory
\endbibitem

%b33 ###
%b32 #&#
\bibitem{newman}
\begin{barticle}[auto:STB|2014/06/10|07:15:57]
\bauthor{\bsnm{Newman},~\bfnm{C.}\binits{C.}}
(\byear{1988}).
\btitle{Memory capacity in neural network models: Rigorous lower bounds}.
\bjournal{Neural Networks}
\bvolume{1}
\bpages{223--238}.
\end{barticle}
\bptok{imsref}%
% NOT OUTPUTED:
%   number = 3
\endbibitem

%b34 ###
%b33 #&#
\bibitem{Petritis}
\begin{bincollection}[mr]
\bauthor{\bsnm{Petritis},~\bfnm{Dimitri}\binits{D.}}
(\byear{1996}).
\btitle{Thermodynamic formalism of neural computing}.
In \bbooktitle{Dynamics of Complex Interacting Systems ({S}antiago, 1994)}.
\bseries{Nonlinear Phenom. Complex Systems}
\bvolume{2}
\bpages{81--146}.
\blocation{Dordrecht}:
\bpublisher{Kluwer Academic}.
\bid{doi={10.1007/978-94-017-1323-8_3}, mr={1421794}}
\end{bincollection}
\bptok{imsref}%
% NOT OUTPUTED:
%   url = http://dx.doi.org/10.1007/978-94-017-1323-8_3
\endbibitem

%b35 ###
%b34 #&#
\bibitem{Schreiber}
\begin{barticle}[auto:STB|2014/06/10|07:15:57]
\bauthor{\bsnm{Piekniewski},~\bfnm{F.}\binits{F.}} \AND
\bauthor{\bsnm{Schreiber},~\bfnm{T.}\binits{T.}}
(\byear{2008}).
\btitle{Spontaneous scale-free structure of spike flow graphs in recurrent neural networks}.
\bjournal{Neural Networks}
\bvolume{21}
\bpages{1530--1536}.
\end{barticle}
\bptok{imsref}%
% NOT OUTPUTED:
%   number = 10
\endbibitem

%b36 ###
%b35 #&#
\bibitem{Rubinov}
\begin{barticle}[auto:STB|2014/06/10|07:15:57]
\bauthor{\bsnm{Rubinov},~\bfnm{M.}\binits{M.}},
\bauthor{\bsnm{McIntosh},~\bfnm{A.}\binits{A.}},
\bauthor{\bsnm{Valenzuela},~\bfnm{M.}\binits{M.}} \AND
\bauthor{\bsnm{Breakspear},~\bfnm{M.}\binits{M.}}
(\byear{2009}).
\btitle{Simulation of neuronal death and network recovery in a computational model of distributed cortical activity}.
\bjournal{Am. J. Geriatr. Psychiatry}
\bvolume{17}
\bpages{210--217}.
\end{barticle}
\bptok{imsref}%
% NOT OUTPUTED:
%   number = 3
\endbibitem

%b37 ###
%b36 #&#
\bibitem{talagrand}
\begin{barticle}[mr]
\bauthor{\bsnm{Talagrand},~\bfnm{Michel}\binits{M.}}
(\byear{1998}).
\btitle{Rigorous results for the {H}opfield model with many patterns}.
\bjournal{Probab. Theory Related Fields}
\bvolume{110}
\bpages{177--276}.
\bid{doi={10.1007/s004400050148}, issn={0178-8051}, mr={1609015}}
\end{barticle}
\bptok{imsref}%
% NOT OUTPUTED:
%   issn = 0178-8051
%   url = http://dx.doi.org/10.1007/s004400050148
%   number = 2
%   coden = PTRFEU
%   fjournal = Probability Theory and Related Fields
\endbibitem

%b38 ###
%b37 #&#
\bibitem{talabook}
\begin{bbook}[mr]
\bauthor{\bsnm{Talagrand},~\bfnm{Michel}\binits{M.}}
(\byear{2003}).
\btitle{Spin Glasses: A Challenge for Mathematicians. Cavity and Mean Field Models}.
\bseries{Ergebnisse der Mathematik und Ihrer Grenzgebiete. 3 Folge. A Series of Modern Surveys in Mathematics [Results in Mathematics and Related Areas. 3rd Series. A Series of Modern Surveys in Mathematics]}
\bvolume{46}.
\blocation{Berlin}:
\bpublisher{Springer}.
\bid{mr={1993891}}
\end{bbook}
\bptok{imsref}%
% NOT OUTPUTED:
%   isbn = 3-540-00356-8
%   fpage = x+586
\endbibitem

%b39 ###
%b38 #&#
\bibitem{vdH}
\begin{bmisc}[auto:STB|2014/06/10|07:15:57]
\bauthor{\bsnm{van~der Hofstadt},~\bfnm{R.}\binits{R.}}
(\byear{2013}).
\bhowpublished{Random graphs and complex networks.
Lecture notes. Available at \url{http://www.win.tue.nl/\textasciitilde rhofstad/NotesRGCN.pdf}.}
\end{bmisc}
\bptok{imsref}%
% NOT OUTPUTED:
%   sortkey = van(2013
\endbibitem

%b40 ###
%b39 #&#
\bibitem{ZhouLipowsky}
\begin{barticle}[auto:STB|2014/06/10|07:15:57]
\bauthor{\bsnm{Zhou},~\bfnm{H.}\binits{H.}} \AND
\bauthor{\bsnm{Lipowsky},~\bfnm{R.}\binits{R.}}
(\byear{2005}).
\btitle{Dynamic pattern evolution on scale-free networks}.
\bjournal{Proc. Natl. Acad. Sci. USA}
\bvolume{102}
\bpages{10052--10057}.
\end{barticle}
\bptok{imsref}%
% NOT OUTPUTED:
%   number = 29
\endbibitem
\end{thebibliography}

% imsref loaded by akundreckaite, 2014-06-11 13:47:16

%
%\begin{appendix}
%\section{}
%\end{appendix}

% zodis "Acknowledgments" paliekamas pagal autoriu
%\section*{Acknowledgements}

%\begin{supplement}%[id=suppA]
%\sname{Supplement A}
%\stitle{}
%\slink[doi]{10.3150/00-BEJXXXXSUPP} %[doi,text={...}] - jei reikia
%suskaldyti doi
%\sdatatype{.pdf}
%\sfilename{BEJ000\_supp.pdf}
%\sdescription{}
%\end{supplement}

%\begin{thebibliography}{00}
%\bibitem[\protect\citeauthoryear{}{()}]{r1}
%\bibitem{r1}
%\end{thebibliography}

\printhistory
\end{document}